\documentclass[11pt]{article}

 \usepackage{amsmath,amssymb,amscd,amsthm,esint}
 
\usepackage{graphics,amsmath,amssymb,amsthm,mathrsfs}

\usepackage{graphics,amsmath,amssymb,amsthm,mathrsfs,amsfonts}

\usepackage{sidecap}
\usepackage{float}
\usepackage{extarrows}
\usepackage{booktabs}
\usepackage{verbatim}
\usepackage{hyperref}
\usepackage[usenames,dvipsnames]{xcolor}

\newtheorem{thm}{Theorem}[section]
\newtheorem{lemma}[thm]{Lemma}

\theoremstyle{definition}
\newtheorem{remark}[thm]{Remark}

\def\XXint#1#2#3{{\setbox0=\hbox{$#1{#2#3}{\int}$}
         \vcenter{\hbox{$#2#3$}}\kern-.5\wd0}}

\def\e{\varepsilon}

\def\loc{\text{loc}}

\numberwithin{equation}{section}

\begin{document}

\title{Weighted $L^2$ Estimates for Elliptic Homogenization\\
in Lipschitz Domains}

\author{
Zhongwei Shen\thanks{Supported in part by NSF grant DMS-1856235.}}
\date{}
\maketitle

\begin{abstract}

We develop a new real-variable method for weighted $L^p$ estimates.
 The  method is applied to the study of  weighted $W^{1, 2}$ estimates in Lipschitz domains for
weak solutions of second-order elliptic systems in divergence form with bounded measurable coefficients.
It produces a necessary and sufficient condition, which depends on the weight function,
for the weighted $W^{1,2}$ estimate to hold in a fixed  Lipschitz domain with a  given weight.
Using this condition, for elliptic systems in Lipschitz domains with rapidly oscillating, periodic and VMO  coefficients,
we reduce the problem of weighted estimates  to the case of constant coefficients.

\medskip


\medskip

\noindent {\it MR (2020) Subject Classification}: 35B27; 35J57; 42B37.

\end{abstract}

\section{Introduction}\label{section-1}

In this paper we are interested in weighted $L^2$ estimates for  the Dirichlet problem,
\begin{equation}\label{DP-0}
\left\{
\aligned
-\text{\rm div} \big(A(x/\e)\nabla u_\e \big)  & =\text{\rm div} (f)  & \quad & \text{ in } \Omega,\\
u_\e & =0 & \quad & \text{ on } \partial\Omega,
\endaligned
\right.
\end{equation}
where $\e>0$ and $\Omega$ is a bounded Lipschitz domain in $\mathbb{R}^d$.
Throughout we assume that the matrix (tensor) $A(y)=\big(a^{\alpha\beta} _{ij} (y)\big)$,
with $1\le i, j\le d$ and $1\le \alpha, \beta\le m$, is real-valued and
 satisfies   the ellipticity condition,
\begin{equation}\label{ellipticity}
\mu |\xi|^2 \le a^{\alpha \beta}_{ij}(y) \xi^\alpha_i \xi^\beta_j \quad 
\text{ and }
\quad
\| A\|_\infty \le \mu^{-1},
\end{equation}
(the summation convention is used), 
 for a.e. $y\in \mathbb{R}^d$ and all $\xi=(\xi^\alpha_i) \in \mathbb{R}^{m\times d}$, where $\mu>0$.
 We also assume that $A$ is 1-periodic; i.e.,
 \begin{equation}\label{periodicity}
A(y+z) =A(y) \quad 
\text{ for }  y \in \mathbb{R}^d \text{ and } z\in \mathbb{Z}^d.
\end{equation}
 By the energy estimate, 
 $
 \| \nabla u_\e \|_{L^2(\Omega)}
 \le \mu^{-1}  \|f \|_{L^2(\Omega)}.
 $
 It was shown in \cite{Shen-2008-W, Shen-2017-W}
  that if $A$ is continuous (or in  VMO$(\mathbb{R}^d)$) and $\Omega$ is a bounded  $C^1$ domain, then
 \begin{equation}\label{W-p}
 \|\nabla u_\e \|_{L^p(\Omega)}
 \le C_p\,  \| f\|_{L^p(\Omega)}
 \end{equation}
 for $1<p< \infty$, where $C_p$ depends on $d$,  $m$, $p$, $A$ and $\Omega$ 
 (see \cite{Auscher-2002, Byun-Wang-2004} for the case $\e=1$).
 In fact, given an exponent $p>2$,  an elliptic matrix $A$ and a bounded Lipschitz domain $\Omega$,
 a necessary and sufficient condition for the $W^{1, p}$ estimate 
 $\|\nabla u\|_{L^p(\Omega)} \le C_p \| f\|_p$ for
 the Dirichlet problem (\ref{DP-1})  was found by the present author in \cite{Shen-2005}.
 This  condition is given in terms of a (weak) reverse H\"older inequality for local  solutions
 of $\text{\rm div} (A\nabla u)=0$.
 Consequently, for the scalar case $m=1$,
 it was proved in \cite{Shen-2008-W} that if $A\in \text{VMO}(\mathbb{R}^d)$ and is symmetric and $\Omega$ is Lipschitz, then
 (\ref{W-p}) holds for $(3/2)-\gamma < p< 3+\gamma $ if  $d\ge 3$, and for 
 $(4/3)-\gamma  < p< 4+\gamma $ if $d=2$, where $\gamma >0$ depends on $\Omega$.
 The ranges of $p$'s are known to be sharp for Lipschitz domains, even in the case of the Laplacian \cite{Kenig-J-1995}.
 For $m\ge 2$, partial results may be found in \cite{Geng-Song-Shen-2012}.
 Also see related work in \cite{Geng-2012} for the Neumann problem.
 
 In this paper we investigate the weighted  $L^2$ estimate,
 \begin{equation}\label{w-e}
 \int_{\Omega} |\nabla  u_\e |^2 \, \omega \, dx 
 \le C_\omega  \int_\Omega |f|^2\,  \omega \, dx
 \end{equation}
 for solutions of  (\ref{DP-0}).
 Using results in \cite{Shen-2005, Shen-2017-W},
it is not hard to see that if $A$ is in VMO$(\mathbb{R}^d)$ and $\Omega$ is a bounded $C^1$ domain, then the inequality
 (\ref{w-e}) holds if either $\omega$ or $\omega^{-1}$ is an $A_1$ weight.
 We point out that the weighted $L^2$ estimate is closely related to the $W^{1, p}$ estimate
 (\ref{W-p}).
 In particular,  for a given $A$ and  a fixed $\Omega$,
 if  (\ref{w-e}) holds for all $\omega$ in the $A_1$ class, then
  (\ref{W-p}) holds for all $2<p<\infty$, by a general extrapolation result of Rubio de Francia (see e.g. \cite{D-book}).
 In view of  this  close connection, we shall not be interested in conditions for which the weighted estimate (\ref{w-e})
 holds for all weights in the $A_1$ class.
 In fact, such conditions may be found in \cite{Shen-2005, Auscher-2008}.
  Rather, in this paper,
 we shall address the question: 
 {\it Given an $A_1$  weight  $\omega$, an elliptic matrix $A$ and a bounded Lipschitz domain $\Omega$, find 
 a necessary and sufficient condition, which may depend on $\omega$, $\Omega$ and $A$, for  the weighted norm inequality (\ref{w-e})}.
 
 The following two theorems are   the main results of the paper.
 
 \begin{thm}\label{main-thm-1}
 Let $\omega$ be an $A_1$ weight in $\mathbb{R}^d$ and $\Omega$  a bounded Lipschitz domain.
 Let  $A$ a matrix satisfying the ellipticity condition (\ref{ellipticity}).
 The following are equivalent.
 
 \begin{enumerate}
 
 \item
 
 For any $f\in L^\infty (\Omega)$, the weak solution in $H^1_0(\Omega)$  of the Dirichlet problem,
 \begin{equation}\label{DP-1}
 -\text{\rm div}
 \big( A\nabla u) =\text{\rm div}(f)
 \quad \text{ in }  \Omega \quad
 \text{ and } \quad
 u=0 \quad \text{ on } \partial\Omega,
 \end{equation}
  satisfies the estimate,
  \begin{equation}\label{w-e-1}
  \int_\Omega |\nabla u|^2 \omega\, dx \le C_1 \int_\Omega |f|^2\omega\, dx.
  \end{equation}
 
 \item
 
 Let $B=B(x_0, r)$, where either $4B \subset \Omega$, or
 $x_0 \in \partial\Omega$ and $0< r< c_0\,  \text{\rm diam}(\Omega)$.
 Let $u\in H^1(4B\cap \Omega)$ be a weak solution of
 $\text{\rm div} (A\nabla u)=0$ in $4B\cap \Omega$ with
 $u=0$ on $4B \cap \partial\Omega$ (in the case $x_0 \in \partial\Omega)$.
 Then
 \begin{equation}\label{cond-1}
 \fint_{B\cap\Omega }
 |\nabla u|^2\,  \omega \, dx
 \le C_2 \fint_{2B \cap \Omega}
 |\nabla u|^2\, dx
 \fint_{B}\omega.
 \end{equation}
 
 \end{enumerate}
 \end{thm}
 
 \begin{thm}\label{main-thm-2}
 Let $\omega$ be an $A_1$ weight  and  $\Omega$  a bounded Lipschitz domain.
Let  $A$  a matrix satisfying (\ref{ellipticity}), (\ref{periodicity}) and $A\in \text{\rm VMO} (\mathbb{R}^d)$.
 Suppose that the inequality  (\ref{w-e-1}) holds for weak solutions  in $H^1_0(\Omega)$ of
 $-\text{\rm div}(\overline{A} \nabla u)=\text{\rm div}(f)$ in $\Omega$, where $f\in L^\infty(\Omega)$ and
 the constant matrix $\overline{A}$ is either the homogenized matrix of
 $A$ or obtained from $A$ by averaging over a ball.
 Then the  weighted  inequality (\ref{w-e}) holds, uniformly in $\e>0$, for any weak solution of (\ref{DP-0}).
  \end{thm}
 
 For any  fixed $A_1$ weight $\omega$, any fixed Lipschitz domain $\Omega$ and any given  elliptic matrix $A$,
 Theorem \ref{main-thm-1} gives  a necessary and sufficient condition for the weighted norm inequality (\ref{w-e-1}).
 To the author's best knowledge, this condition (\ref{cond-1}), which depends on the weight $\omega$,
   is new even for the Laplacian. 
 Theorem \ref{main-thm-2} reduces the weighted estimate for the elliptic operator 
$\mathcal{L}_\e =-\text{\rm div}(A(x/\e)\nabla )$ with rapidly oscillating coefficients 
to the same estimate for elliptic operators with constant coefficients.
By combining these two theorems we see that, to establish the weighted $L^2$ estimate (\ref{w-e}) for
the operator $\mathcal{L}_\e$,
it suffices to verify the condition (\ref{cond-1}) for local solutions of
$\text{\rm div}(\overline{A}\nabla u)=0$, where $\overline{A}$ is either the homogenized matrix of $A$ or
obtained from $A$ by averaging over a ball.

 One of our motivations for studying (\ref{w-e})  lies in a  special case,
 \begin{equation}\label{sw}
 \omega=\omega_\sigma (x)  =[\text{\rm dist}  (x, \partial\Omega) ]^\sigma.
 \end{equation}
 Note that $\omega_\sigma \in A_1(\mathbb{R}^d)$ if  $-1< \sigma\le 0$ and $\Omega$ is Lipschitz.
 The weighted inequality (\ref{w-e}) for this  special case arises in the study of the quantitative  homogenization theory  
 and provides  useful estimates for boundary layers   \cite{KLS-2012} as well
 as control  of solutions at infinite for unbounded domains \cite{Shen-Zhuge-2018}.
As a corollary of Theorem \ref{main-thm-2}, we obtain the following.

\begin{thm}\label{main-thm-3}
Let $\Omega$ be a bounded Lipschitz domain in $\mathbb{R}^d$.
Suppose  that $A$  satisfies conditions (\ref{ellipticity})-(\ref{periodicity}) and that
$A\in \text{\rm VMO}(\mathbb{R}^d)$.
In the case $m\ge 2$ we also assume that $A^*=A$, i.e., $a_{ij}^{\alpha\beta}
=a_{ji}^{\beta\alpha}$.
Let $u_\e$ be a weak solution of (\ref{DP-0}).
Then for any $-1 < \sigma< 1$, 
\begin{equation}\label{main-3}
\int_\Omega 
|\nabla u_\e |^2 \big[ \text{\rm dist} (x, \partial\Omega) \big]^\sigma \, dx
\le C_\sigma
\int_\Omega
| f|^2 \big[ \text{\rm dist} (x, \partial\Omega) \big]^\sigma \, dx,
\end{equation}
where  
$C_\sigma$ depends only on $d$, $m$, $A$,  $\sigma$ and the Lipschitz character of $\Omega$.
\end{thm}

\begin{remark}\label{remark-0}
Consider the scalar case $m=1$.
It follows from \cite{Shen-2005, Shen-2008-W} 
 that if $A$ is in VMO$(\mathbb{R}^d)$ and $\Omega$ is Lipschitz,
then (\ref{w-e}) holds for $\omega=( \widetilde{\omega}) ^\sigma$, where $\widetilde{\omega} \in A_1(\mathbb{R}^d)$ and
$|\sigma|<\frac{1}{3} +\gamma $ for $d\ge 3$, and $|\sigma| < \frac12 +\gamma  $ for $d=2$,
where $\gamma>0$ depends on $\Omega$.
The ranges of $\sigma$  are sharp for Lipschitz domains.
This result would only yield  (\ref{main-3}) for $|\sigma |< \frac13 + \gamma$ if $d\ge 3$, and
for $|\sigma|< \frac12 +\gamma$ if $d=2$.
Thus, even though the weighted inequality (\ref{w-e}) may not be true for all weights in the $A_1$ class,
the inequality (\ref{main-3})  still holds in Lipschitz domains for the full range of  possible $\sigma$'s.
We also note that without any smoothness and periodicity  conditions on $A$,
 (\ref{main-3})  holds for $|\sigma|<  \kappa$, where
 $\kappa>0$ depends on $d$, $m$, $\mu$ and the Lipschitz character of  $\Omega$.
 See Theorem \ref{g-thm}.

\end{remark}

Our approach to Theorems \ref{main-thm-1} and \ref{main-thm-2} is based on a new real-variable 
method for establishing weighted $L^2$ estimates.
In \cite{Shen-2005, Shen-2006, Shen-2007-L} we developed a real-variable method for establishing $L^p$
estimates (also see related work in \cite{Auscher-2007}).
The method, which is originated in \cite{CP-1998} (also see \cite{Wang-2003, Byun-Wang-2004}),
is particularly effective in the non-smooth settings, where the $L^p$ estimates are  expected  only for $p$
in some limited ranges.
The basic idea  is that  to prove the $L^p$ estimate  for a function $F$, where $p>2$,
for each small ball $B$, one decomposes $F$ into two parts,  $F_B$ and $R_B$, both depending on $B$.
For $F_B$, one establishes an $L^2$ estimate with a small parameter $\eta$ for the term involving $F$.
For $R_B$  one proves an $L^q$ estimate for some $q>p$ .
We point out that the proof is based on a good-$\lambda$ inequality.
As such, there is a direct extension of this method to the weighted setting, which has been exploited in \cite{Shen-2005, Auscher-2008}.
The main novelty of this paper is that instead of requiring an $L^q$ estimate with $q>p$ for the function $R_B$,
we require $R_B$ to satisfy a weighted estimate in $L^{p_1}$ for some $p_1>2$.
See Theorems \ref{real-thm} and \ref{real-thm-b}.
As a result, instead of conditions for weighted estimates for a whole class of weights,
our conditions for weighted estimates are for each individual weight $\omega$.
In particular, we remark that the  result in Theorem \ref{main-thm-3} does  not  seem to be accessible by the 
methods used in \cite{Shen-2005, Auscher-2008}.
We expect this insight to be useful in other related problems.
 
 The paper is organized as follows.
 In Section \ref{section-real} we present a general real-variable method, described above,
 for weighted $L^p$  estimates, where $0<p<\infty$.
 See Theorem \ref{real-thm}. 
  In Section \ref{section-real-R} we apply the real-variable method in Section \ref{section-real}
 to sublinear operators in $\mathbb{R}^d$, including linear operators of Calder\'on-Zygmund type.
In Section \ref{section-main-1} we use a boundary version of Theorem \ref{real-thm} for  a Lipschitz domain 
to prove  Theorem \ref{main-thm-1}.
Sections \ref{section-small} and \ref{section-large}
contain  the proof of Theorem \ref{main-thm-2}.
We point out that the small parameter $\eta$ in Theorems \ref{real-thm} and \ref{real-thm-b} is particularly useful for
perturbation arguments.
For local estimates ($\e=1$), as in the study  of $W^{1, p}$ estimates,
 a  perturbation argument reduces the case of VMO coefficients to the case
of constant coefficients.
See Section \ref{section-small}.
For large-scale estimates in homogenization, a similar perturbation with the use of the parameter $\eta$ as well as
convergence rates   allows us to reduce the problem to the same estimates for  the homogenized operator.
See Section \ref{section-large}.
Finally, Theorem \ref{main-thm-3} is proved in Section \ref{section-sp}.
We remark that the periodic structure of $A$  is not essential for Theorems \ref{main-thm-2} and
\ref{main-thm-3}, as long as $|\nabla \chi|\in L^q_{\loc}$ uniformly for any $q>2$,
where $\chi$ denotes the corrector.

We will use $C$ and $c$ to denote constants that may depend on $d$, $m$, $\mu$, and the Lipschitz 
character of $\Omega$.
If a constant also depends on other parameters, it will be stated explicitly.
We use $\fint_E u$ to denote the average of $u$, with respect to the Lebesgue measure,  over the set $E$; i.e. 
$$
\fint_E u =\frac{1}{|E|} \int_E u.
$$
  
%%%%%%%%%%%%%%%%%%%%%%%%%%%%

\section{A real-variable method for weighted  estimates}\label{section-real}

We begin with a brief review of $A_p$ weights and refer the reader to \cite{D-book} for a detailed presentation.
For $1<p<\infty$,
a nonnegative  function  $\omega\in L^1_{\loc} (\mathbb{R}^d)$ is called an $A_p$ weight, 
denoted by $\omega\in A_p (\mathbb{R}^d)$,  if
there exists a constant  $C_\omega \ge 1$ such that
\begin{equation}\label{A-p}
\fint_B \omega
\cdot \left(\fint_B \omega^{-\frac{1}{p-1}} \right)^{p-1}
\le C_\omega \quad \text{ for any ball  } B \subset\mathbb{R}^d.
\end{equation}
In the case $p=1$, the condition (\ref{A-p}) is replaced by
\begin{equation}\label{A-1}
\fint_{B} \omega \le C_\omega  \inf_B \omega \quad \text{ for any ball } B \subset \mathbb{R}^d.
\end{equation}
We will refer to the smallest $C_\omega$ for which  (\ref{A-p}) (or  (\ref{A-1}) for $p=1$) holds as the $A_p$ constant of $\omega$.
It follows by  H\"older's inequality that
$A_p(\mathbb{R}^d)\subset A_q(\mathbb{R}^d)$ if $p<q$.
It is also known that if $\omega\in A_q(\mathbb{R}^d)$ for some $q>1$,
then $\omega\in A_p(\mathbb{R}^d) $ for some $p<q$ ($p$ depends on $\omega$).
A function $\omega$ is called an $A_\infty$ weight if it is an $A_p$ weight for some
$p\ge 1$.
An $A_\infty$ weight satisfies  the doubling condition,  
\begin{equation}\label{doubling}
\omega(2B) \le C\,  \omega (B) \quad \text{  for any ball }
B\subset \mathbb{R}^d,
\end{equation}
 where we have used the notation 
 $\omega(E)= \int_E \omega$.
Moreover,  if   $\omega$ is an $A_\infty$ weight, then there exist $\sigma \in (0, 1)$ and $C>0$ such that
\begin{equation}\label{A-infty}
\frac{\omega (E)}{\omega(B)}
\le C \left(\frac{|E|}{|B|}\right)^\sigma,
\end{equation}
where $E\subset B$ is measurable and $B$ is a ball, and that
\begin{equation}\label{RH}
\left(\fint_B \omega^{1+\sigma} \right)^{\frac{1}{1+\sigma}}
\le C \fint_B \omega.
\end{equation}
Using the doubling condition, it is not hard to see that
balls $B$ in  (\ref{A-p}), \ref{A-1}), (\ref{doubling}), (\ref{A-infty}) and (\ref{RH}) 
may be replaced by cubes $Q$.

The goal of this section is to prove the following theorem.

\begin{thm}\label{real-thm}
Let $0< p_0<p<p_1<\infty$ and $\omega$ be an $A_{p/p_0}  $ weight in $\mathbb{R}^d$.
Let $F\in L^p(4B_0)$ and $f\in L^p(4B_0)$,
where $B_0$ is a ball in $\mathbb{R}^d$.
Suppose that for each ball $B\subset 2B_0$ with $|B|\le c_1 |B_0|$,
there exist two functions $F_B$ and $R_B$,  defined on $2B$, such that
$|F|\le |F_B| + |R_B|$ on $2B$, and that
\begin{align}
\left(\fint_{2B} |F_B|^{p_0}\right)^{1/p_0}
 & \le N_1 \sup _{4B_0\supset B^\prime \supset B}
\left(\fint_{B^\prime} |f|^{p_0}  \right)^{1/p_0}
+ \eta \left(\fint_{100B} |F|^{p_0} \right)^{1/p_0}, \label{real-1}\\
\left(\fint_{2B}
|R_{B}|^{p_1} \omega\, dx \right)^{1/p_1}
 & \le N_2
\left\{ 
\left(\fint_{100B} |F|^{p_0}\right)^{1/p_0}
+ \sup_{4B_0 \supset B^\prime\supset B}
\left(\fint_{B^\prime}
|f|^{p_0}\right)^{1/p_0} \right\}
\left(\fint_{B}  \omega \right)^{1/p_1},
\label{real-2}
\end{align}
where  $N_1, N_2>1$, $0<c_1<1$,
and $\eta\ge 0$.
Then there exists $\eta_0>0$, depending
only on $d$, $p_0$, $p_1$, $p$,  $c_1$,  $N_2$, and the $A_{p/p_0}$ constant of $\omega$, with the property that if
$0\le \eta\le  \eta_0$, then
\begin{equation}\label{real-3}
\left( 
\fint_{B_0} |F|^p \omega\, dx \right)^{1/p}
\le C \left(\fint_{4B_0} |f|^p \omega\right)^{1/p}
+  C \left(\fint_{4B_0} |F|^{p_0} \right)^{1/p_0}
\left( \fint_{B_0} \omega \right)^{1/p},
\end{equation}
where $C$ depends only on $d$, $p_0$, $p_1$,  $p$, $c_1$, $N_1$, $N_2$, and the $A_{p/p_0}$ constant of $\omega$.
\end{thm}

For the most part,
the argument for Theorem \ref{real-thm} is parallel  to that for Theorem 3.1 in \cite{Shen-2005}.
It starts with a Calder\'on-Zgymund decomposition given in the next lemma.

\begin{lemma}\label{CZ-lemma}
Let $Q$ be a cube in $\mathbb{R} ^d$.
Suppose that $E\subset Q$ is open and $|E|< 2^{-d} |Q|$.
Then there exists a sequence $\{ Q_k\}$
of disjoint dyadic subcubes of $Q$ such that,

\begin{enumerate}

\item

$Q_k\subset E$;

\item

the dyadic parent of $Q_k$ in $Q$ is not contained in $E$;

\item

$ |E\setminus \cup_k Q_k| =0$.

\end{enumerate}
\end{lemma}

\begin{proof}
See e.g. \cite[p.75]{Shen-book}.
\end{proof}

For $f\in L^1_{\loc} (\mathbb{R}^d)$ and a ball $B\subset \mathbb{R}^d$,
define
\begin{equation}\label{maximal}
\mathcal{M}_B (f) (x)=
\sup \left\{
\fint_{B^\prime} |f|: \ x\in B^\prime \text{ and } B^\prime \subset B \right\}
\end{equation}
for $x\in B$.

\begin{lemma}\label{w-lemma}
Let $\omega$ be an $A_1$ weight.
Then 
\begin{equation}\label{maximal-1}
\omega \big\{ x\in B: \  \mathcal{M}_B (f)(x)>\lambda \big\}
  \le \frac{C}{\lambda} \int_{B} |f|\omega\, dx
  \end{equation}
  for any $\lambda>0$,
  where $C$ depends only on $d$ and $C_\omega$ in  (\ref{A-1}). 
   If $1<p<\infty$ and $\omega\in A_p(\mathbb{R}^d)$, then 
  \begin{equation}\label{maximal-2}
\int_{B } |\mathcal{M}_{B} (f)  |^p \omega \, dx 
\le C \int_{B} |f|^p \omega\, dx,
 \end{equation}
 where $C$ depends at most on $d$, $p$ and  $C_\omega $ in (\ref{A-p}).
 \end{lemma}
 
 \begin{proof}
 Observe that $\mathcal{M}_B (f) \le \mathcal{M}(f\chi_B)$ on $B$, where
 $\mathcal{M}$ is the  Hardy-Littlewood maximal operator in $\mathbb{R}^d$.
  As a result, (\ref{maximal-1}) and (\ref{maximal-2}) follow
 from the standard weighted norm inequalities for the operator $\mathcal{M}$
 with $A_p$ weights.
 See e.g. \cite{D-book}.
 \end{proof}
 
\begin{proof}[\bf Proof of Theorem \ref{real-thm}]

Let $Q_0$ be a cube such that $2Q_0\subset 2B_0$ and $|Q_0|\approx |B_0|$.
We shall  show that
\begin{equation}\label{real-4}
\left( 
\fint_{Q_0} |F|^p \omega\, dx \right)^{1/p}
\le C \left(\fint_{4B_0} |f|^p \omega\right)^{1/p}
+ C \left(\fint_{4B_0} |F|^{p_0}\right)^{1/p_0}
\left( \fint_{B_0} \omega \right)^{1/p}.
\end{equation}
The inequality (\ref{real-3}) follows from (\ref{real-4}) by a simple covering argument.

For $t>0$, let  
\begin{equation}\label{E-t}
E(t)= \big\{ x\in Q_0: \ \mathcal{M}_{4B_0} (|F|^{p_0})(x)  > t \big\}.
\end{equation}
We claim that  if $0\le \eta\le  \eta_0$ and $\eta_0>0$ is sufficiently small,
then it is possible to choose three constants $\theta, \gamma \in (0, 1)$ and $C_0>0$, such that
\begin{equation}\label{real-5}
\omega (E(\theta^{-p_0/p} t))
\le (\theta/2)  \omega(E(t))
+ \omega \big\{ x\in Q_0: \mathcal{M}_{4B_0} (|f|^{p_0} ) (x) >\gamma t \big \}
\end{equation}
for all $t\ge t_0$, where
\begin{equation}\label{real-6}
t_0= C_0 \fint_{4B_0} |F|^{p_0} .
\end{equation}
Moreover,  the constants $\eta_0$,  $\theta$ and $C_0$
depend at most on $d$, $p_0$, $p_1$, $p$,  $c_1$, $N_2$, and
the $A_{p/p_0}$ constant of $\omega$. The constant $\gamma$  also depends on $N_1$.

Assume (\ref{real-5}) for a moment.
Then
$$
\aligned
 & \int_{t_0}^T  t^{\frac{p}{p_0}-1}
\omega(E(\theta^{-p_0/p} t))\, dt\\
&\quad
 \le \frac{\theta}{2}
\int_{t_0}^T t^{\frac{p}{p_0}-1} \omega (E(t))\, dt
+ \int_{t_0}^T
t^{\frac{p}{p_0}-1}
\omega \big\{ x\in Q_0:
\mathcal{M}_{4B_0} (|f|^{p_0} )>\gamma t\big\}\, dt
\endaligned
$$
for any $T>t_0$. By a change of variables we obtain 
$$
\theta
\int_{\theta^{-\frac{p}{p_0}} t_0}^{\theta^{-\frac{p}{p_0}}T}
t^{\frac{p}{p_0} -1} 
 \omega(E( t))\, dt
\le \frac{\theta}{2}
\int_{t_0}^T t^{\frac{p}{p_0}-1}  \omega (E(t))\, dt
+  C \int_{Q_0}
 \big(\mathcal{M}_{4B_0} (|f|^{p_0}) \big)^{p/p_0}
 \omega\, dx.
$$
It follows that
$$
\frac{\theta}{2}
\int_{\theta^{-\frac{p}{p_0}} t_0}^{\theta^{-\frac{p}{p_0}}T}
t^{\frac{p}{p_0}-1} 
\omega (E(t))\, dt
\le \frac{\theta}{2}
\int_0^{\theta^{-\frac{p}{p_0}} t_0}
t^{\frac{p}{p_0}-1}
\omega (E(t))\, dt 
+ C \int_{Q_0}
\big( \mathcal{M}_{4B_0} (|f|^{p_0}) \big)^{p/p_0} \omega\, dx .
$$
This leads to
$$
\int_0^{\theta^{-\frac{p}{p_0} }T}
t^{\frac{p}{p_0}-1}
\omega (E(t))\, dt
\le C t^{\frac{p}{p_0}}_0 \omega(Q_0)
+ C\int_{Q_0}
\big( \mathcal{M}_{4B_0} (|f|^{p_0}) \big)^{p/p_0} \omega\, dx.
$$
By letting  $T\to \infty$ we see that
\begin{equation}\label{real-7}
\aligned
\int_{Q_0}
\big( \mathcal{M}_{4B_0} (|F|^{p_0} )\big)^{p/p_0}  \omega\, dx
 & \le 
C t^{\frac{p}{p_0}}_0 \omega(Q_0)
+ C \int_{Q_0}
\left( \mathcal{M}_{4B_0} (|f|^{p_0} )\right)^{p/p_0} \omega\, dx\\
&\le C\left(  \fint_{4B_0} |F|^{p_0} \right)^{p/p_0} \omega(B_0)
+ C \int_{4B_0} |f|^p \, \omega \, dx,
\endaligned
\end{equation}
where we have used (\ref{real-6}), the assumption $\omega\in A_{p/p_0}(\mathbb{R}^d)$ 
and   (\ref{maximal-2}) with $p=\frac{p}{p_0}>1$ for the last inequality.
The inequality (\ref{real-4})  now follows readily from (\ref{real-7}).

To see (\ref{real-5}), we first note that 
$$
|E(t)|\le \frac{C_d}{t}\int_{4B_0} |F|^{p_0} \, dx
< \theta |Q_0|
$$
if $ t> t_0$, where $t_0$ is given by (\ref{real-6}) with
$$
 C_0=2\theta^{-1} C_d |4B_0|/|Q_0|
 $$
  and $\theta\in (0, 2^{-d} )$ is a small constant to be determined.
 Fix $t>t_0$.
 Since $E(t)$ is open in $Q_0$, by Lemma \ref{CZ-lemma}, there exists a sequence  $\{ Q_k\}$
 of non-overlapping maximal  dyadic subcubes of $Q_0$ such that
 $$
 \cup_k Q_k \subset E(t)  \quad \text{ and } \quad 
 |E(t)\setminus \cup_k Q_k |=0.
 $$
 By choosing $\theta$ small, we may assume that $|Q_k|\le c_1 |Q_0|$.
 We shall show that if $0\le \eta\le \eta_0$ and $\eta_0>0$ is sufficiently small, then
 one may choose $\theta, \gamma \in (0, 2^{-d})$ such that 
 \begin{equation}\label{real-9}
 \omega (E(\theta^{-\frac{p_0}{p}} t)\cap Q_k )
 \le (\theta/2) \omega (Q_k),
 \end{equation}
 whenever
 \begin{equation}\label{real-10}
 Q_k \cap \big\{ x\in Q_0: \mathcal{M}_{4B_0} (|f|^{p_0}) (x) \le \gamma t  \big\}
 \neq \emptyset.
 \end{equation}
 It is not hard to see that this yields (\ref{real-5}).
 
 Finally, to prove (\ref{real-9}),
 we use the observation  that for any $x\in Q_k$,
 \begin{equation}
 \mathcal{M}_{4B_0} (|F|^{p_0}) (x) 
 \le \max \left (\mathcal{M}_{2B_k} (|F|^{p_0} ) (x) , C_d t \right),
 \end{equation} 
 where $B_k$ is the smallest ball containing $Q_k$ and $C_d$ depends only on $d$.
 It follows that if $\theta^{-\frac{p_0}{p}} > C_d$, 
 $$
 \aligned
 \omega (E(\theta^{-\frac{p_0}{p}} t)\cap Q_k)
 & \le \omega
\left \{ x\in Q_k:\  \mathcal{M}_{2B_k} (|F|^{p_0} ) (x) >\theta^{-\frac{p_0}{p}}  t\right\}\\
 &\le \omega \left\{ x\in Q_k: \  \mathcal{M}_{2B_k}
 (|F_{B_k} |^{p_0}) (x)>\frac{\theta^{-\frac{p_0}{p}} t }{2^{p_0+1}} \right\}\\
& \qquad
 + \omega \left\{ x\in Q_k:  \  \mathcal{M}_{2B_k}
 (|R _{B_k} |^{p_0}) (x)>\frac{\theta^{-\frac{p_0}{p}}
 t }{2^{p_0+1}} \right\}\\
 &=I_1 +I_2.
 \endaligned
 $$
 To bound $I_1$,
 we let
 $$
 E_k =\left\{ x\in Q_k:  \ \mathcal{M}_{2B_k}
 (|F_{B_k} |^{p_0} ) (x)>\frac{\theta^{-\frac{p_0}{p}} t }{2^{p_0+1}} \right\}.
 $$
 By  using (\ref{maximal-1}) with $\omega=1$, the assumption (\ref{real-1}) and (\ref{real-10}), wee see that
 $$
 \aligned
 |E_k|
 &  \le \frac{C\theta^{\frac{p_0}{p}} |Q_k|}{t}
 \fint_{2B_k}  |F_{B_k} |^{p_0}\, dx\\
 &\le C \theta ^{\frac{p_0}{p}}( N^{p_0}_1 \gamma  +\eta^{p_0} )  |Q_k|,
 \endaligned
 $$
where $C$ depends only on $d$ and $p_0$.
It follows  from (\ref{A-infty}) that
$$
\aligned
\frac{\omega(E_k)}{\omega (Q_k)}
 & \le C\left(\frac{|E_k|}{|Q_k|}\right)^\sigma\\
 &\le C \theta^{\frac{\sigma p_0}{p}}
 N_1^{p_0 \sigma} \gamma^\sigma 
 + C \eta^{p_0 \sigma} \theta^{\frac{\sigma p_0}{p}},
 \endaligned
$$
where $\sigma \in (0, 1)$ and $C>0$ depend only on $d$, $p_0$ and the $A_{p/p_0}$ constant of $\omega$.
Hence,
\begin{equation}\label{real-20}
I_1\le 
\left\{ 
 C \theta^{\frac{\sigma p_0}{p}}
 N_1^{p_0 \sigma} \gamma^{ \sigma}
 + C \eta^{p_0 \sigma} \theta^{\frac{\sigma p_0}{p}}\right\}
  \omega (Q_k).
 \end{equation}
 
To estimate $I_2$, we use (\ref{maximal-2}) with $p=\frac{p_1}{p_0}>1$  and the assumption (\ref{real-2})  to obtain
$$
\aligned
I_2 &\le  C
\bigg(
\frac{\theta^{\frac{p_0}{p}}}{t}
\bigg)^{\frac{p_1}{p_0}}
\fint_{2B_k}
|R_{B_k}|^{p_1}\, \omega\, dx\,  |Q_k|\\
&\le 
 C
 \bigg(
\frac{\theta^{\frac{p_0}{p}}}{t}
\bigg)^{\frac{p_1}{p_0}}
\left( N_2^{p_1}  t^{p_1/p_0} + \gamma^{p_1/p_0}  t^{p_1/p_0} \right)
\omega (Q_k)\\
&\le C \theta^{\frac{p_1}{p}}
N^{p_1}_2 \omega(Q_k),
\endaligned
$$
where $C$ depends only on $d$, $p_0$, $p_1$ and the $A_{p/p_0}$ constant of $\omega$.
This, together with (\ref{real-20}), gives
\begin{equation}\label{real-21}
 \omega (E(\theta^{-\frac{p_0}{p}} t)\cap Q_k)
 \le 
\left( C \theta^{\frac{\sigma p_0}{p}}
 N_1^{p_0 \sigma} \gamma^{\sigma}
 + C \eta^{p_0\sigma} \theta^{\frac{p_0 \sigma}{p}}
 +C \theta^{\frac{p_1}{p}}
N_2^{p_1}
 \right) \omega (Q_k),
\end{equation}
where $C>0$ and $\sigma>0$ depend at most  on $d$, $p_0$, $p_1$, $p$ and the $A_{p/p_0}$ constant $C_\omega$ in (\ref{A-p}).
To conclude  the proof,
we choose $\theta\in (0, 2^{-d})$ so small that 
$$
C\theta^{\frac{p_1}{p}}N_2^{p_1} < (1/6) \theta.
$$
This is possible since $p_1>p$.
With $\theta$ chosen, we choose $\gamma>0$ so small that 
$$
C \theta^{\frac{\sigma p_0}{p}}  N_1^{p_0\sigma} \gamma^{\sigma}  \le (1/6) \theta.
$$
Finally, we choose $\eta_0>0$ so that
$$
C\eta_0^{p_0 \sigma} \theta^{\frac{p_0\sigma}{p}} \le (1/6) \theta.
$$
It follows that if $0\le \eta\le \eta_0$, then (\ref{real-9} holds.
We note  that the small positive constants $\theta$ and $\eta_0$ depend 
at most on $d$,  $c_1$, $p_0$, $p_1$, $p$,   $ N_2$  and the $A_{p/p_0}$ constant for the weight $\omega$,
and that $\gamma$   also depends on $N_1$.
This completes the proof.
\end{proof}

\begin{remark}\label{remark-L}

Let $\e>0$.
Define
\begin{equation}\label{maximal-e-1}
{ \mathcal{M}}^{\e}(f) (x)=
\sup \left\{
\fint_{B^\prime} |f|: \ x\in B^\prime=B(y, r) \text{ and } r\ge \e,   \right\},
\end{equation}
and for $x\in B$.
\begin{equation}\label{maximal-e}
{ \mathcal{M}}^{\e}_B (f) (x)=
\sup \left\{
\fint_{B^\prime} |f|: \ x\in B^\prime=B(y, r)\subset B \text{ and }  r\ge \e  \right\}.
\end{equation}
Let $E_\e (t)$ be defined as in the proof of Theorem \ref{real-thm}, but $\mathcal{M}_{4B_0}$ replaced
by $\mathcal{M}^\e_{4B_0}$; i.e.,
\begin{equation}\label{E-t-e}
E_\e (t)= \left\{ x\in Q_0: \ \mathcal{M}^\e_{4B_0} (|F|^{p_0})(x)  > t \right\}.
\end{equation}
Observe that if $Q$ is a cube and $Q\subset E_\e(t)$, then 
$|Q|\ge c_d \e^d$, where $c_d$ depends only on $d$.
It follows that the sequence $\{Q_k\}$ in the proof of Theorem \ref{real-thm}
satisfies $|Q_k |\ge c_d \e^d$. Since $Q_k \subset B_k$,
we also obtain $|B_k|\ge c_d \e^d$.
Consequently, if the conditions (\ref{real-1}) and (\ref{real-2}) are only satisfied  by 
balls $B$ with radius $r\ge \e$, then 
\begin{equation}\label{real-3-e}
\aligned
& \left(\fint_{B_0}
\big\{ \mathcal{M}_{4B_0}^\e (|F|^{p_0})  \big\}^{\frac{p}{p_0}}  \omega \right)^{1/p}\\
 & \le C \left(\fint_{4B_0} |f|^p \omega \right )^{1/p}
+C \left(\fint_{4B_0} |F|^{p_0}\right)^{1/p_0}
\left(\fint_{B_0} \omega \right)^{1/p},
\endaligned
\end{equation}
where $C$ is independent of $\e$.
This inequality will be used for large-scale estimates in homogenization in Section \ref{section-large}.
We mention that in the case of $L^p$ estimates, a similar  observation was made in \cite{Armstrong-L-p}.
\end{remark}

%%%%%%%%%%%%%%%%%%%%%%%%%%%%%%%%

\section{Weighted $L^p$ estimates in $\mathbb{R}^d$}\label{section-real-R}

In this section we use Theorem \ref{real-thm} to establish weighted norm inequalities 
for operators in $\mathbb{R}^d$.
An operator $T$ is called sublinear if
there exists a constant $K\ge 1$ such that
\begin{equation}\label{sublinear}
|T(f+g)|\le K \big\{ |T(f)| + |T(g)| \big\}.
\end{equation}
We use $L^\infty_c(\mathbb{R}^d)$ to denote the space of bounded measurable functions with compact support in $\mathbb{R}^d$.

The following theorem may be regarded as the weighted version of Theorem 3.1 in \cite{Shen-2005}.
We emphasize that the condition (\ref{real-2-1}) and hence the conclusion (\ref{real-2-2})   are for
an individual weight.

\begin{thm}\label{real-thm-2}
Let $T$ be a sublinear operator in $\mathbb{R}^d$.
Let $0< p_0< p< p_1<\infty$ and $\omega$ be an $A_{p/p_0}$ weight.
Suppose that
\begin{enumerate}

\item

$T$ is bounded on $L^{p_0}(\mathbb{R}^d)$  with
$
\| T\|_{L^{p_0} \to L^{p_0}} \le C_0$;

\item 

for any ball $B\subset\mathbb{R}^d$ and any $g\in L_c^\infty (\mathbb{R}^d)$ with {\rm supp}$(g)\subset \mathbb{R}^d\setminus 4B$,
one has
\begin{equation}\label{real-2-1}
\left(\fint_{B} |T(g)|^{p_1} \omega \right)^{1/p_1}
\le N 
\left\{ 
\left(\fint_{3B} |T(g)|^{p_0} \right)^{1/p_0}
+\sup_{B^\prime\supset B}
\left(\fint_{B^\prime} |g|^{p_0} \right)^{1/p_0} \right\}
\left(\fint_{B} \omega\right)^{1/p_1}.
\end{equation}

\end{enumerate}
Then for any $f\in L_c^\infty (\mathbb{R}^d)$,
\begin{equation}\label{real-2-2}
\int_{\mathbb{R}^d}
|T(f)|^p \omega \, dx
\le C \int_{\mathbb{R}^d} |f|^p \omega \, dx,
\end{equation}
where $C$ depends only on $d$,  $K$, $p_0$, $p_1$, $p$, $C_0$, $N$, and the $A_{p/p_0}$
constant of $\omega$.
\end{thm}

\begin{proof}
Let $f\in L_c^\infty(\mathbb{R}^d)$. 
Fix $B_0=B(0, R)$, where $R>1$ is so large that supp$(f) \subset B(0, R/4)$.
We apply Theorem \ref{real-thm} with $f$ and $F=T(f)$.
For each ball $B\subset 2B_0$ with $|B|\le c_d |B_0|$, we define 
$$
F_B= K T(f\varphi_B) \quad 
\text{ and } 
\quad
R_B =K T (f(1-\varphi_B)),
$$
where $\varphi_B \in C_0^\infty(9B)$ satisfies  $0\le \varphi \le 1$ and $\varphi_B =1$ on $8B$.
Note that by (\ref{sublinear}), $|F|\le |F_B| + |R_B|$ on $\mathbb{R}^d$, and that 
$$
\aligned
\left(\fint_{2B}
|F_B|^{p_0} \right)^{1/p_0}
&=K \left(\frac{1}{|2B|}
\int_{\mathbb{R}^d}
|T(f\varphi_B)|^{p_0}\, dx \right)^{1/p_0}\\
& \le C \left(\fint_{9B} |f|^{p_0} \right)^{1/p_0},
\endaligned
$$
where we have used the boundedness  of $T$ on $L^{p_0}(\mathbb{R}^d) $ for the last inequality.
Next, since $f(1-\varphi_B)=0$ in $8B$,
it follows from the assumption (\ref{real-2-1}) that
$$
\aligned
\left(\fint_{2B} |R_B|^{p_1} \omega \right)^{1/p_1}
&\le 
C \left\{
\left(\fint_{6B} |R_B|^{p_0} \right)^{1/p_0}
+ \sup_{4B_0\supset B^\prime\supset B}
\left(\fint_{B^\prime}
|f|^{p_0} \right)^{1/p_0} 
\right\}
\left(\fint_{2B} \omega \right)^{1/p_1}\\
& 
\le 
C \left\{
\left(\fint_{6B} |F|^{p_0} \right)^{1/p_0}
+ \sup_{4B_0\supset B^\prime\supset B}
\left(\fint_{B^\prime}
|f|^{p_0} \right)^{1/p_0} 
\right\}
\left(\fint_B \omega \right)^{1/p_1},
\endaligned
$$
where we have used the fact $|R_B|\le K (|F| +|F_B|)$ and the boundedness of $T$ on $L^{p_0}(\mathbb{R}^d)$.
Thus we have verified the conditions in Theorem \ref{real-thm} with $\eta=0$, which yields 
$$
\left(\fint_{B_0}
|T(f)|^p \omega \right)^{1/p}
\le C \left(\fint_{4B_0}
|f|^p \omega \right)^{1/p}
+ C
\left(\fint_{4B_0} |T(f)|^{p_0} \right)^{1/p_0}
\left(\fint_{B_0} \omega\right)^{1/p}.
$$
Hence,
\begin{equation}\label{real-2-4}
\int_{B_0}
|T(f)|^p \omega\, dx
\le C \int_{\mathbb{R}^d}  |f|^p \omega\, dx
+ C |B_0|^{1-\frac{p}{p_0}}
\|T(f)\|_{L^{p_0}(\mathbb{R}^d)}^p
\fint_{B_0} \omega.
\end{equation}
Finally, we note that the condition  $\omega\in A_{p/p_0}(\mathbb{R}^d)$ implies 
$\omega\in A_q(\mathbb{R}^d)$ for some $1<q< \frac{p}{p_0}$.
It follows that
$$
\fint_{B_0} \omega
\le {C |B_0|^{q-1}}{ \left( \int_{B_0} \omega^{-\frac{1}{q-1}}\right)^{1-q}  }
\le  C_\omega  |B_0|^{q-1},
$$
where $C_\omega$ depends on $\omega$.
Therefore, by letting $R\to \infty$ in (\ref{real-2-4}),
we obtain (\ref{real-2-2}).
\end{proof}

\begin{remark}
If $T$ is a Calder\'on-Zygmund operator,
then it satisfies the condition (\ref{real-2-1}).
Indeed, suppose 
$$
T(f)(x) =\int_{\mathbb{R}^d}
K(x, y) f(y)\, dy
$$
for $x\notin \text{supp}(f)$, where $K(x, y)$ satisfies the condition
$$
| K(x+h, y)-K(x, y)| \le \frac{C |h|^\sigma}{|x-y|^{d+\sigma}}
$$
for any $x, y, h \in \mathbb{R}^d$ with $|h|< (1/2)|x-y|$, where $\sigma\in (0, 1]$.
Let $g\in L_c^\infty(\mathbb{R}^d)$ with supp$(g) \subset \mathbb{R}^d\setminus 4B$.
Then 
$$
\| T(g)\|_{L^\infty(B)}
\le \fint_{B} |T(g)|
+ C \sup_{B^\prime\supset B} \fint_{B^\prime} |g|,
$$
which shows that (\ref{real-2-2}) holds for any $1\le p_0<p_1<\infty$.
Thus, if $T$ is bounded on $L^{p_0}(\mathbb{R}^d)$ for some $1< p_0<\infty$,
it is bounded on $L^p(\mathbb{R}^d, \omega dx )$ for any $p_0<p<\infty$ and
$\omega\in A_{p/p_0} (\mathbb{R}^d)$.
\end{remark}

Consider the elliptic system
\begin{equation}\label{ES}
-\text{\rm div} (A\nabla u) =\text{\rm div} (f) \quad \text{ in } \mathbb{R}^d,
\end{equation}
where $A=(a_{ij}^{\alpha\beta}(x) )$
satisfies the ellipticity condition (\ref{ellipticity}).
Given $f\in L_c^\infty(\mathbb{R}^d)$, there exists $u\in H^1_{\loc}(\mathbb{R}^d)$ 
that satisfies (\ref{ES}) and $|\nabla u|\in L^2(\mathbb{R}^d)$.
Moreover, the solution $u$ is unique up to a constant and
$\|\nabla u\|_{L^2(\mathbb{R}^d)}
\le C \| f\|_{L^2(\mathbb{R}^d)}$.
The next theorem gives a sufficient (and necessary) condition for the weighted $L^2$ inequality,
\begin{equation}\label{ES-1}
\int_{\mathbb{R}^d} |\nabla u|^2 \omega \, dx 
\le C \int_{\mathbb{R}^d} |f|^2 \omega\, dx.
\end{equation}

\begin{thm}\label{ES-thm}
Let $A$ be a matrix satisfying (\ref{ellipticity}).
Let $ \omega$ be an $A_1$ weight.
Suppose that for any ball $ B \subset \mathbb{R}^d$,
\begin{equation}\label{ES-0}
\fint_{B} |\nabla u|^{2} \omega 
\le N  \fint_{2B} |\nabla u|^2
\cdot  
\fint_{B} \omega,
\end{equation}
whenever  $u\in H^1(4B)$ is a weak solution of $\text{\rm div}(A\nabla u)=0$ in $4B$.
Then, for any $f\in L_c^\infty(\mathbb{R}^d)$, weak solutions of (\ref{ES}) with $|\nabla u|\in L^2(\mathbb{R}^d)$
satisfy the estimate (\ref{ES-1}), 
where $C$ depends only on $d$, $m$, $\mu$, $N$ and the $A_1$ constant of $\omega$.
\end{thm}

\begin{proof}
For $f\in L_c^\infty(\mathbb{R}^d)$,
let $T(f)=\nabla u$, where $u$ is a weak solution of (\ref{ES}) with $|\nabla u|\in L^2(\mathbb{R}^d)$.
To prove (\ref{ES-1}), we shall apply Theorem \ref{real-thm-2} with $p=2$.
First, by the Myers estimates,   there exists $1< p_0<2$, depending only on $d$, $m$ and $\mu$, such that
$T$ is bounded on $L^{p_0} (\mathbb{R}^d)$.
Next, let $g\in L_c^\infty(\mathbb{R}^d)$ with supp$(g)\subset \mathbb{R}^d\setminus 4B$ and $v=T(g)$.
Then $\text{\rm div}(A\nabla v) =0$ in $4B$.
By the reverse H\"older estimates,
$$
\left(\fint_{B^\prime }|\nabla v|^2 \right)^{1/2}
\le C 
\fint_{(3/2)B^\prime} |\nabla v|,
$$
where $B^\prime$ is a ball with $2B^\prime\subset 4B$.
This, together with the condition (\ref{ES-0}), gives
$$
\aligned
\left(\fint_{B^\prime} |(\nabla v) \omega^{1/2}|^2 \right)^{1/2}
&\le C \fint_{3B^\prime} |\nabla v| \cdot \left(\fint_{B^\prime} \omega\right)^{1/2}\\
&\le C \fint_{3B^\prime}
|(\nabla v) \omega^{1/2} |,
\endaligned
$$
where we have used the $A_1$ condition (\ref{A-1}) for the last  inequality.
By the self-improving property of the reverse H\"older inequality
for the function $|(\nabla v) \omega^{1/2}|$,
there exists $p_1>2$, depending only on $d$, $m$,  $\mu$, $N$ and the $A_1$ constant of $\omega$,
such that
$$
\aligned
\left(\fint_{B}
|(\nabla v)\omega^{1/2}|^{p_1} \right)^{1/p_1}
 & \le  C
\left(\fint_{(5/4)B}
|(\nabla v)\omega^{1/2}|^{2} \right)^{1/2}\\
&\le C \left(\fint_{(5/2)B}
|\nabla v|^2\right)^{1/2}
\cdot \left(\fint_{B} \omega\right)^{1/2},
\endaligned
$$
where we have used the condition (\ref{ES-0}) for the last inequality.
Using the $A_1$ condition (\ref{A-1}) again, we obtain 
\begin{equation}\label{SI}
\left(\fint_{B} |\nabla v|^{p_1} \omega \right)^{1/p_1}
\le C  \fint_{3B} |\nabla v| \cdot 
\left(\fint_{B} \omega\right)^{1/p_1}.
\end{equation}
As a result, we have proved that $T$ satisfies the conditions in Theorem \ref{real-thm-2}, from which 
(\ref{ES-1}) follows.
\end{proof}

\begin{remark}

The condition (\ref{ES-0}) in Theorem \ref{ES-thm} is also necessary.
We provide a proof for the case of bounded Lipschitz domains in the next section.
The same argument works equally well for the case of  $\mathbb{R}^d$.

\end{remark}

%%%%%%%%%%%%%%%%%%%%%%%%%%%%%%%%%%%

\section{Boundary weighted estimates}\label{section-main-1}

In this section we present a boundary version of Theorem \ref{real-thm} for a Lipschitz domain
and give the proof of Theorem \ref{main-thm-1}.

\begin{thm}\label{real-thm-b}
Let $0<  p_0<p<p_1<\infty$ and $\omega$ be an $A_{p/p_0}  $ weight.
Let $\Omega$ be a bounded Lipschitz domain and $B_0=B(x_0, r_0)$ where $x_0\in \partial\Omega$ and
$0<r_0< c_0\, \text{\rm diam}(\Omega)$.
Let $F\in L^p(4B_0 \cap \Omega )$ and $f\in L^p(4B_0\cap \Omega)$.
Suppose that for each ball $B=B(y_0, r)$ with the properties that  $|B|\le c_1 r_0^d$ and
either $y_0\in  2B_0\cap \partial \Omega$ or $4B\subset  2B_0\cap \Omega $,
there exist two functions $F_B$ and $R_B$,  defined on $2B\cap \Omega$, such that
$|F|\le |F_B| + |R_B|$ on $2B\cap \Omega $, and
\begin{align}
\left(\fint_{ 2B\cap \Omega } |F_B|^{p_0}\right)^{1/p_0}
 & \le N_1 \sup _{ 4B_0\supset B^\prime \supset B}
\left(\fint_{ B^\prime\cap \Omega} |f|^{p_0}  \right)^{1/p_0}
+ \eta \left(\fint_{8B \cap \Omega} |F|^{p_0} \right)^{1/p_0}, \label{real-1b}\\
\left(\fint_{2B \cap \Omega }
|R_{B}|^{p_1} \omega \right)^{1/p_1}
 & \le N_2
\left\{ 
\left(\fint_{8B\cap \Omega} |F|^{p_0}\right)^{1/p_0}
+ \sup_{ 4B_0\supset B^\prime\supset B}
\left(\fint_{B^\prime\cap \Omega}
|f|^{p_0}\right)^{1/p_0} \right\}
\left(\fint_{B}  \omega \right)^{1/p_1},
\label{real-2b}
\end{align}
where  $N_1, N_2>1$, $0<c_1<1$,
and $\eta\ge 0$.
Then there exists $\eta_0>0$, depending
only on $d$, $p_0$, $p_1$,  $p$, $c_1$,  $N_2$,  the $A_{p/p_0}$ constant of $\omega$ and the Lipschitz character of $\Omega$, with the property that if
$0\le \eta\le  \eta_0$, then
\begin{equation}\label{real-3b}
  \left( 
\fint_{B_0\cap \Omega } |F|^p \omega \right)^{1/p}\\
 \le C \left(\fint_{4B_0\cap \Omega  } |f|^p \omega\right)^{1/p}
+  C \left(\fint_{4B_0\cap \Omega } |F|^{p_0} \right)^{1/p_0}
\left( \fint_{ B_0 } \omega \right)^{1/p},
\end{equation}
where $C$ depends only on $d$, $p_0$, $p_1$, $p$, $c_1$, $N_1$, $N_2$,  the $A_{p/p_0}$ constant of $\omega$
and the Lipschitz character of $\Omega$.
\end{thm}

\begin{proof}
We shall apply Theorem \ref{real-thm} to the functions 
$$
\widetilde{F}=F\chi_{ 4B_0\cap \Omega}
\quad \text{  and }
\quad
\widetilde{f}=f\chi_{ 4B_0 \cap \Omega},
$$
 with $B_0=B(x_0, r_0)$.
Let $B$ be a ball such that $2B\subset 2B_0$ with $|B|\le \widetilde{c}_1 |B_0|$.
We need to construct two functions
$\widetilde{F}_B$ and $\widetilde{R}_B$, which satisfy 
$|\widetilde{F}| \le \widetilde{F}_B +\widetilde{R}_B$ on $2B$ and
the conditions (\ref{real-1}) and (\ref{real-2}).

We consider three cases: (1) $4B\cap \Omega =\emptyset$;
(2) $ 4B\subset \Omega $; and (3) $4B\cap \partial\Omega\neq \emptyset$.
 In the first case we simply define $F_B$ and $R_B$ to be zero.
 In the second case we let $\widetilde{F}_B=F_B$ and $\widetilde{R}_B=R_B$.
 To treat the third case, we assume $ B=B(z, r)$ and $y_0\in 4B\cap \partial\Omega$.
 Let $\widetilde{B}=B(y_0, 4r)$, 
 $$
 \widetilde{F}_B= F_{\widetilde{B}}\chi_{ 4B_0 \cap \Omega}
\quad
\text{ and }
\quad
\widetilde{R}_B= R_{\widetilde{B}}\chi_{ 4B_0 \cap \Omega }.
$$ 
 Since $2B\subset 2\widetilde{B}$ and $2\widetilde{B}\subset 12B$,
 it is not hard to verify the conditions in Theorem \ref{real-thm}. 
 As a result, the inequality \ref{real-3b}) follows from (\ref{real-3}).
\end{proof}

\begin{remark}\label{remark-L-b}
Let $0<\e< r_0$.
Suppose that conditions (\ref{real-1b}) and (\ref{real-2b}) hold only for balls $B$
with radius $r\ge \e$.
Then
\begin{equation}\label{real-3b-e}
\aligned
 &  \left( 
\fint_{ B_0\cap \Omega  } \big\{ \mathcal{M}_{4B_0}^\e (|F|^{p_0} \chi_{ 4B_0\cap \Omega} )  \big\}^{\frac{p}{p_0}} \omega \right)^{1/p}\\
 &  \le C \left(\fint_{ 4B_0\cap \Omega } |f|^p \omega\right)^{1/p}
+  C \left(\fint_{ 4B_0\cap \Omega } |F|^{p_0} \right)^{1/p_0}
\left( \fint_{ B_0 } \omega \right)^{1/p},
\endaligned
\end{equation}
where the operator $\mathcal{M}_{4B_0}^\e$ is defined by
(\ref{maximal-e}) and  $C$ is independent of $\e$.
See Remark \ref{remark-L}
\end{remark}

The rest of this section is devoted to the proof of Theorem \ref{main-thm-1}.
We begin by  giving  a condition equivalent to  (2) in Theorem \ref{main-thm-1}.

\begin{lemma}\label{lemma-m-1}
Let $\omega$ be an $A_1$ weight and  $\Omega$ a bounded Lipschitz domain.
Let $A$ a matrix satisfying (\ref{ellipticity}).
The following reverse H\"older condition is equivalent to (2) in Theorem \ref{main-thm-1}:  There exist
$p>2$ and $C_3>0$ such that
\begin{equation}\label{eqv-1}
\left(\fint_{B(x_0, r)\cap \Omega} 
| ( \nabla u) \omega^{1/2} |^{p} \right)^{1/p}
\le C_3  \fint_{B(x_0, 2r)\cap \Omega} | (\nabla u) \omega^{1/2}|,
\end{equation}
where  $u$ and $B(x_0, r)$ are  the same as in Theorem \ref{main-thm-1}.
\end{lemma}

\begin{proof}
We first note that (\ref{cond-1})  follows readily from (\ref{eqv-1}) by using H\"older's inequality.
To see that  (\ref{cond-1}) implies (\ref{eqv-1}),
we consider the case $x_0\in \partial\Omega$.
The interior  case $B(x_0, 4r)\subset \Omega$ may be  handled similarly. 
Let  $B^\prime=B(y_0, t)$ be a ball such that $y_0\in B(x_0, r)\cap \Omega$ and $0<t<cr$.
Then
$$
\aligned
\left(\fint_{B^\prime\cap \Omega}
|(\nabla u) \omega^{1/2} |^2 \right)^{1/2}
&\le C \left(\fint_{8B^\prime\cap \Omega} |\nabla u|^2 \right)^{1/2}
\left(\fint_{B^\prime} \omega \right)^{1/2}\\
&\le C \fint_{9B^\prime\cap \Omega} |\nabla u| 
\left(\fint_{B^\prime} \omega \right)^{1/2}\\
&\le C  \fint_{9B^\prime\cap \Omega}
|(\nabla u)\omega^{1/2}|,
\endaligned
$$
where we have used (\ref{A-1}) for the last inequality, and $C$ depends only on $d$, $m$, $\mu$, $C_2$ ,
the $A_1$ constant of $\omega$ and the Lipschitz
character of $\Omega$.
This is a reverse H\"older inequality for the function $|(\nabla u)\omega^{1/2}|$.
By the self-improving property of such inequalities,
there exist $p>2$ and $C_3>0$, depending at most on 
$d$, $m$, $\mu$, $C_2$, the $A_1$ constant of $\omega$ and the Lipschitz character of $\Omega$, such that (\ref{eqv-1}) holds.
\end{proof}

\begin{lemma}\label{lemma-m-2}
Condition (2) implies Condition (1) in Theorem \ref{main-thm-1}.
\end{lemma}

\begin{proof}

Let $u\in H_0^1(\Omega)$ be a weak solution of (\ref{DP-1}) with $f\in L^\infty (\Omega)$.
To show (\ref{w-e-1}), we fix $x_0\in \partial \Omega$, $0<r_0=c_0\,  \text{\rm diam}(\Omega)$, and
apply Theorem \ref{real-thm-b} with  $p=2$, $f$ and $F=|\nabla u|$.
Let $B=B(y_0, r) $ be a ball with the properties  that $|B|\le c_1 r_0^d$ and either
$y_0\in B(x_0, 2r_0) \cap \partial \Omega$ or $4B\subset  B(x_0, 2r_0)\cap \Omega$.
Let $v\in H^1_0(\Omega)$ be the weak solution of
$-\text{\rm div}(A\nabla v)=\text{\rm div}( f\varphi)$ in $\Omega$, where
$\varphi\in C_0^\infty(5B)$ is a cut-off function such that
$0\le \varphi\le 1$ and $\varphi=1$ on $4B$.
Let
$$
F_B= |\nabla v| \quad \text{ and } \quad 
R_B = |\nabla (u-v)|.
$$
By the Myers estimate, there exists $1<p_0<2$, depending only on $d$, $m$, $\mu$ and the Lipschitz character of $\Omega$, such that
\begin{equation}\label{myers-5}
\| \nabla v\|_{L^{p_0}(\Omega)}
\le C \| \varphi f\|_{L^{p_0}(\Omega)}
\le C \|f\|_{L^{p_0}(5B\cap \Omega)},
\end{equation}
which yields (\ref{real-1b}) with $\eta=0$.
To see (\ref{real-2b}), we observe that
$\text{\rm div}(A\nabla (u-v))=0$ in $4B\cap \Omega$
and $u-v=0$ on $\partial\Omega$.
It follows from Lemma \ref{lemma-m-1} that 
$$
\aligned
\left(\fint_{ 2B\cap \Omega}
|(\nabla (u-v) ) \omega^{1/2} )|^{p_1} \right)^{1/p_1}
&
\le C \fint_{4B\cap \Omega}
|(\nabla (u-v))\omega^{1/2}|\\
 & \le C \left(\fint_{4B\cap \Omega}
|\nabla (u-v)|^{p_0} \right)^{1/p_0}
\left(\fint_{B} \omega\right)^{1/2},
\endaligned
$$
where $p_1>2$ and we have used (\ref{RH}) for the last inequality.
Hence,
$$
\aligned
\left(\fint_{ 2B\cap \Omega}
|R_B|^{p_1} \omega\right)^{1/p_1}
&\le C \left(\fint_{ {4B\cap \Omega}}
|\nabla (u-v)|^{p_0}  \right)^{1/p_0} 
\left(\fint_B \omega\right)^{1/p_1}\\
&\le 
C \left\{
\left(\fint_{ 4B\cap \Omega}
|F|^{p_0}
\right)^{1/p_0}
+
\left(\fint_{4B\cap \Omega } 
|\nabla v|^{p_0} \right)^{1/p_0}
\right\}
\left(\fint_B \omega\right)^{1/p_1}\\
&
\le 
C \left\{
\left(\fint_{ 4B\cap \Omega}
|F|^{p_0}
\right)^{1/p_0}
+
\left(\fint_{5B\cap \Omega } 
|f|^{p_0} \right)^{1/p_0}
\right\}
\left(\fint_B \omega\right)^{1/p_1},
\endaligned
$$
where we have used (\ref{myers-5}) for the last inequality.
Thus, by Theorem \ref{real-thm-b}, we obtain
$$
\aligned
\int_{ B(x_0, r_0)\cap \Omega}
|\nabla u|^2 \omega\, dx
& \le C\int_{ B(x_0, 4r_0)\cap \Omega} 
|f|^2 \omega\, dx
+ \int_{\Omega} |\nabla u|^2\, dx \fint_{B(x_0, r_0)} \omega\\
&\le C \int_\Omega |f|^2 \omega\, dx,
\endaligned
$$
where we have used the energy estimate and (\ref{A-1}) for the last inequality.
A similar argument shows that if $B(x_0, 4r_0)\subset \Omega$, where
$r_0=c_0\, \text{diam}(\Omega)$, then
$$
\int_{B(x_0, r_0)} |\nabla u|^2\omega\, dx
\le C \int_\Omega |f|^2\omega\, dx.
$$
By a simple covering argument we obtain (\ref{w-e-1}),
where $C_1$ depends only on $d$, $m$, $\mu$, $C_2$, the $A_1$ constant of $\omega$ and
the Lipschitz character of $\Omega$.
\end{proof}

We are now in a position to give the proof of Theorem \ref{main-thm-1}.

\begin{proof}[\bf Proof of Theorem \ref{main-thm-1}]

In view of Lemma \ref{lemma-m-2},  it remains to show that Condition (1) implies Condition (2).
Suppose $d\ge 3$.
Let $A^*$ denote the adjoint of $A$ and $\omega$ be an $A_1$ weight.
It follows by a duality argument from  (\ref{w-e-1}) that
 weak solutions in $H^1_0(\Omega)$  of $-\text{\rm div}(A^*\nabla v)=\text{\rm div} (g)$ in $\Omega$
satisfy
\begin{equation}\label{d-estimate}
\int_{\Omega}
|\nabla v|^2  \, \frac{dx}{\omega}
\le C \int_{\Omega} |g|^2\,  \frac{dx}{\omega},
\end{equation}
where $g\in L^\infty(\Omega) $.
This, together with  the weighted Sobolev inequality,
\begin{equation}\label{w-s}
\left(\int_{\Omega}
|v|^q\frac{dx}{\omega^{\frac{q}{2}}}
\right)^{1/q}
\le C 
\left(\int_{\Omega}
|\nabla v|^2
\frac{dx}{w }\right)^{1/2},
\end{equation}
where  $v\in H^1_0(\Omega)$ and $q=\frac{2d}{d-2}$, yields
$$
\left(\int_{\Omega}
|v|^q\frac{dx}{\omega^{\frac{q}{2}}}
\right)^{1/q}
\le C 
\left(\int_{\Omega}
|g|^2
\frac{dx}{w }\right)^{1/2}.
$$
We remark that the inequality (\ref{w-s}) holds for any $A_p$ weight with $p=2-\frac{2}{d}$ (see \cite{Sawyer-1992} and
earlier work in \cite{Kenig-1982}) . 
It follows by duality that weak solutions in $H^1_0 (\Omega)$  of $-\text{\rm div}(A\nabla u) =F$ in $\Omega$ satisfy
\begin{equation}\label{ES-00}
\left( \int_{\Omega} |\nabla u|^2 \omega\, dx \right)^{1/2}
\le C \left( \int_{\Omega} |F|^{q^\prime}  \omega^{\frac{q^\prime }{2}} \, dx \right)^{1/q^\prime},
\end{equation}
where $F\in L^\infty(\Omega) $  and $q^\prime =\frac{2d}{d+2}$.
Thus, if $u\in H^1_0(\Omega)$ is a weak solution of
\begin{equation}\label{DF}
-\text{\rm div}(A\nabla u) =\text{\rm div} (f) +F \quad \text{ in } \Omega,
\end{equation}
where $F, f \in L^\infty(\Omega)$,  then
\begin{equation}\label{w-e-r-0}
\left(\int_{\Omega}
|\nabla u|^2\omega\, dx \right)^{1/2}
\le C \left(\int_{\Omega}
|f|^2\omega\, dx \right)^{1/2}
+
C \left(\int_{\Omega} 
|F|^{q^\prime} \omega^{\frac{q^\prime}{2}}\, dx \right)^{1/q^\prime}.
\end{equation}
Let $f, F\in L^2(\Omega)$ such that the right-hand side of (\ref{w-e-r-0}) is finite.
By a density argument one may show that (\ref{w-e-r-0})  continues to hold for 
the weak solution $u$ in $H^1_0(\Omega)$  of (\ref{DF}).

In the case $d=2$ we replace the inequality (\ref{w-s}) by
\begin{equation}\label{w-s-2}
\left(\int_{ B(x_0, r)\cap \Omega}
 |v|^q \frac{dx}{\omega^{\frac{q}{2}} }\right)^{1/q}
\le C r^{\frac{2}{q} } \left(\int_\Omega |\nabla v|^2\frac{dx}{\omega} \right)^{1/2}
\end{equation}
for $v\in C_0^1(\Omega)$ and $x_0\in \overline{\Omega}$,
which holds for any $2<q<\infty$ and  $\omega \in A_1(\mathbb{R}^2)$.
It follows by duality  that if supp$(F)\subset B(x_0,r)$, the weak solution of
$-\text{\rm div}(A\nabla u)=F$ in $\Omega$ with $u=0$ on $\partial\Omega$
satisfies 
\begin{equation}\label{ES-00-2}
\left( \int_{\Omega} |\nabla u|^2 \omega\, dx \right)^{1/2}
\le C r^{2-\frac{2}{q^\prime}}
 \left( \int_{ B(x_0, r)\cap \Omega}
  |F|^{q^\prime}  \omega^{\frac{q^\prime }{2}} \, dx \right)^{1/q^\prime}.
\end{equation}
As a result, the estimate (\ref{w-e-r-0}) is replaced by
\begin{equation}\label{w-e-r-2}
\left(\int_{\Omega}
|\nabla u|^2\omega\, dx \right)^{1/2}
\le C \left(\int_{\Omega}
|f|^2\omega\, dx \right)^{1/2}
+
Cr^{2-\frac{2}{q^\prime}}
 \left(\int_{ B(x_0, r)\cap \Omega}
|F|^{q^\prime} \omega^{\frac{q^\prime}{2}}\, dx \right)^{1/q^\prime}.
\end{equation}

Finally, to prove (\ref{cond-1}),
 we let   $u\in H^1(4B\cap \Omega)$ be a weak solution of $-\text{\rm div} (A\nabla u)=0$ in $4B\cap\Omega$,
 with $u=0$ on $4B\cap \partial\Omega$, where $B=B(x_0, r)$, $0<r<c_0\,  \text{diam}(\Omega)$, and 
 either $x_0\in \partial\Omega$ or $4B\subset \Omega$.
For $1<s<t<2$,
let $\varphi\in C_0^\infty(tB)$ be a cut-off function such that
$0\le \varphi\le 1$, $\varphi=1$ on $sB$, and $|\nabla \varphi|\le C ((t-s)r) ^{-1}$.
Note that
$$
-\text{\rm div}(A\nabla ((u-k)  \varphi))
=-\text{\rm div}(A(\nabla \varphi) (u-k) )
-A\nabla u\cdot \nabla \varphi,
$$
where $k\in \mathbb{R}$.
It follows from (\ref{w-e-r-0})  that if $d\ge 3$,
$$
\aligned
& \left(\int_{sB\cap \Omega}
|\nabla u|^2\omega\, dx \right)^{1/2}\\
 & \le \frac{C}{(t-s) r}
 \left\{ 
\left(\int_{tB\cap \Omega} |u-k|^2\omega\, dx \right)^{1/2}
+
\left(\int_{tB\cap \Omega}
|\nabla u|^{\frac{2d}{d+2}}
\omega^{\frac{d}{d+2}} \, dx \right)^{\frac{d+2}{2d}}\right\} \\
& \le \frac{C}{(t-s) r}
\left(\int_{tB\cap \Omega}
|\nabla u|^{\frac{2d}{d+2}}
\omega^{\frac{d}{d+2}} \, dx \right)^{\frac{d+2}{2d}},
\endaligned
$$
where we have let  $k=\fint_{tB} u$ and used  the weighted Sobolev inequality (see \cite{Kenig-1982, Sawyer-1992}),
\begin{equation}\label{w-s-l}
\left(\int_{tB\cap \Omega} |u-k|^2\omega\, dx \right)^{1/2}
\le C 
\left(\int_{tB\cap \Omega}
|\nabla u|^{\frac{2d}{d+2}}
\omega^{\frac{d}{d+2}} \, dx \right)^{\frac{d+2}{2d}}
\end{equation}
for $\omega\in A_1(\mathbb{R}^d)$.
As a result, we obtain 
\begin{equation}\label{nec-1}
\left(
\fint_{sB\cap\Omega}
|(\nabla u) \omega^{1/2}|^2 \right)^{1/2}
\le 
\frac{C}{t-s}
\left(\fint_{tB\cap \Omega}
|(\nabla u) \omega^{1/2}|^\frac{2d}{d+2} \right)^{\frac{d+2}{2d}}
\end{equation}
for $1<s<t<2$.
By a convexity argument this implies that
\begin{equation}\label{nec-2}
\left(
\fint_{B\cap\Omega}
|(\nabla u) \omega^{1/2}|^2 \right)^{1/2}
\le C
\fint_{2B\cap \Omega}
|(\nabla u) \omega^{1/2}|.
\end{equation}
The inequality (\ref{cond-1}) follows from (\ref{nec-2}) by the Cauchy inequality.

In the case $d=2$ we use (\ref{w-e-r-2}) in the place of (\ref{w-e-r-0}) , and 
\begin{equation}\label{w-s-l-2}
\left(\int_{tB\cap \Omega} |u-k|^2\omega\, dx \right)^{1/2}
\le C r^{1-\frac{2}{q^\prime}}
\left(\int_{tB\cap \Omega}
|\nabla u|^{q^\prime}
\omega^{\frac{q^\prime}{2}} \, dx \right)^{1/q^\prime}
\end{equation}
in the place of (\ref{w-s-l}), where $1<q^\prime<2$.
This gives
\begin{equation}\label{nec-1-1}
\left(
\fint_{sB\cap\Omega}
|(\nabla u) \omega^{1/2}|^2 \right)^{1/2}
\le 
\frac{C}{t-s}
\left(\fint_{tB\cap \Omega}
|(\nabla u) \omega^{1/2}|^{q^\prime} \right)^{1/q^\prime}
\end{equation}
for any $1<q^\prime<2$, which leads to (\ref{nec-2}), as in the case $d\ge 3$.
\end{proof}

%%%%%%%%%%%%%%%%%%%%%%%%%%%%%%%%%%%%

\section{Weighted $L^2$ estimates at the small scale}\label{section-small}

In this section we give the proof of Theorem \ref{main-thm-2} for the case $\e=1$.
The periodicity condition (\ref{periodicity})  is not needed.
Recall that a  function $h\in L^1_{\loc}(\mathbb{R}^d)$ is said to belong to VMO$(\mathbb{R}^d)$ if 
$\rho(r; h)\to 0$ as $r\to 0$, where
\begin{equation}\label{VMO}
\rho(r; h)=
\sup_{
\substack{ 
{x\in\mathbb{R}^d}
\\
0<t\le r}}
\fint_{B(x, t)}
\Big |h -\fint_{B(x, t)}
h \Big|.
\end{equation}

\begin{thm}\label{thm-local}
Let $\omega$ be an $A_1$ weight and  
$\Omega$ a bounded Lipschitz domain with $\text{\rm diam}(\Omega)\le 1$.
Let $A$ be a matrix satisfying (\ref{ellipticity}) and $A\in \text{\rm VMO}(\mathbb{R}^d)$.
Suppose that Condition (1) in Theorem \ref{main-thm-1} holds for all constant matrices  $\overline{A}$ 
obtained from $A$ by averaging it over a ball.
Then the condition holds for the matrix $A$.
\end{thm}

\begin{proof}

By Theorem \ref{main-thm-1} it suffices to prove Condition (2).
Let $B_0=B(x_0, r_0)$ be a ball with $x_0\in \overline{\Omega}$ and
$0<r_0< c_0$.
Let $u\in H^1(4B_0\cap \Omega)$ be a weak solution of 
$\text{\rm div}(A\nabla u)=0$ in $4B_0\cap \Omega$ with $u_\e=0$ on $4B_0\cap \partial\Omega$.
We will prove the inequality (\ref{cond-1}).

We consider the case $x_0\in \partial\Omega$ and use Theorem \ref{real-thm-b} with
$
F=|\nabla u|   \text{ and }  f=0
$
on $4B_0\cap \Omega$.
Let $B=B(y_0, r)$ be a ball such that  $|B|\le c_1 |B_0|$ and either $y_0\in 2B_0\cap \Omega$ or
$4B\subset 2B_0\cap \Omega$.
Again, we consider the case $y_0\in 2B_0\cap\partial  \Omega$ (the interior case is similar).
To construct $F_B$ and $R_B$, we let $v\in H^1(3B\cap \Omega)$ be the weak solution
of 
$$
\text{\rm div}(\overline{A}\nabla v)=0 \quad \text{ in } 3B \cap \Omega
\quad 
\text{ and } \quad
v=u \quad 
\text{ on } \partial(3B\cap \Omega),
$$
where
\begin{equation}\label{A-bar}
\overline{A}=\fint_{3B} A.
\end{equation}
Clearly, the constant matrix $\overline{A}$ satisfies the condition  (\ref{ellipticity}).
Define
$$
F_B= |\nabla (u-v) | \quad \text{ and } \quad 
R_B=|\nabla v|.
$$
Note that $u-v=0$ on $\partial(3B\cap \Omega)$ and 
$$
-\text{\rm div} (\overline{A}\nabla (u-v)) 
=\text{\rm div} ((A-\overline{A}) \nabla u) \quad \text{ in } 3B\cap \Omega.
$$
By the Myers estimates, there exist $1<p_0<2$, and $C>0$, depending only on $d$, $m$, $\mu$ and
the Lipschitz character of $\Omega$, such that
$$
\aligned
\fint_{3B\cap \Omega}
|\nabla (u-v) |^{p_0}\, dx
 & \le C \fint_{3B\cap\Omega} |A-\overline{A}|^{p_0}  |\nabla u|^{p_0}\, dx\\
 &\le 
 C \left(\fint_{3B} |A-\overline{A}|^{\frac{2p_0}{2-p_0}}\right)^{\frac{2-p_0}{2}}
 \left(\fint_{3B\cap \Omega}
 |\nabla u|^2 \right)^{\frac{p_0}{2}},
 \endaligned
$$
where we have used H\"older's inequality for the last inequality.
It follows that
\begin{equation}\label{local-10}
\left(\fint_{2B\cap \Omega}
|F_B|^{p_0}\right)^{1/p_0}
\le C\widetilde{\rho} (c_0)
\left(\fint_{4B\cap \Omega}
|F|^{p_0}\right)^{1/p_0},
\end{equation}
where
\begin{equation}\label{rho-t}
\widetilde{\rho}(c_0)=
\sup_{
\substack{ 
{x\in\mathbb{R}^d}
\\
0<t\le c_0}}
\left( \fint_{B(x, t)}
\Big |A -\fint_{B(x, t)} 
A \Big|^{\frac{2p_0}{2-p_0}}
\right)^{\frac{2-p_0}{2p_0}}.
\end{equation}
By the John-Nirenberg  inequality, 
we see that $\widetilde{\rho}(c_0)\to 0$ as $c_0\to 0$.
This implies that the function $F_B$ satisfies the condition (\ref{real-1b}) in Theorem \ref{real-thm-b}
if $c_0>0$ is sufficiently small.

Finally, we note that by the assumption,
Condition (2) in Theorem \ref{main-thm-1} holds for the matrix $\overline{A}$ given by (\ref{A-bar}).
In view of Lemma \ref{lemma-m-1},
there exist $p_1>2$ and $C_3>0$ such that
\begin{equation}\label{local-11}
\left(
\fint_{2B\cap \Omega}
| (\nabla v) \omega^{1/2}|^{p_1} \right)^{1/p_1}
\le  C_3
\fint_{4B\cap \Omega} |(\nabla v) \omega^{1/2}|,
\end{equation}
from which the condition (\ref{real-2b}) in Theorem \ref{real-thm-b} follows readily.
Hence, by Theorem \ref{real-thm-b}, we obtain 
$$
\left(\fint_{B_0\cap \Omega}
|\nabla u|^2\omega \right)^{1/2}
\le C \left(\fint_{4B_0\cap \Omega} 
|\nabla u|^2 \right)^{1/2}
\left(\fint_{B_0} \omega \right)^{1/2},
$$
which is equivalent to (\ref{cond-1}) by s simple covering  argument.
\end{proof}

\begin{remark}\label{remark-D}

Theorem \ref{thm-local} continues to hold if diam$(\Omega)>1$.
However, in this case,
 the constants $C$ will depend on diam$(\Omega)$.

\end{remark}

\begin{remark}\label{remark-C-1}

Let $\Omega$ be a bounded $C^1$ domain.
Then Condition (2) in Theorem \ref{main-thm-1} holds for any $A_1$ weight $\omega$ and
for any matrix $A$ satisfying (\ref{ellipticity}) and $A\in \text{VMO}(\mathbb{R}^d)$.
Indeed, let $B=B(x_0, r)$, where  $x_0\in \partial\Omega$
and $0<r< c_0 \, \text{\rm diam}(\Omega)$. 
Suppose that  $\text{\rm div}(A\nabla u)=0$ in $4B\cap \Omega$ and 
$u=0$ on $4B \cap \partial\Omega$. Then
$$
\left(\fint_{B\cap \Omega}
|\nabla u|^p \right)^{1/p}
\le C_p \left(\fint_{2B\cap \Omega}
|\nabla u|^2 \right)^{1/2}
$$
for any $p>2$.
By using H\"older's inequality and (\ref{RH}) we obtain  (\ref{cond-1}).
Consequently, by  Theorem \ref{main-thm-1}, the weighted  inequality (\ref{w-e}) holds
if $\Omega$ is $C^1$, $\omega\in A_1(\mathbb{R}^d)$ and
$A\in \text{VMO}(\mathbb{R}^d)$  satisfies (\ref{ellipticity}).
This result is not new and was already proved in \cite{Shen-2005}.

\end{remark}

%%%%%%%%%%%%%%%%%%%%%%%%%%%%%%%%%%%%%%%%%%%%%%%%%%%

\section{Large-scale weighted estimates in homogenization }\label{section-large}

In this section we give the proof of Theorem \ref{main-thm-2}.
We begin by introducing some notations as well as some  approximation results  from  the homogenization theory.
We should point out that Lemmas \ref{app-lemma-1}, \ref{app-lemma-3}, \ref{app-lemma-5} and \ref{app-lemma-6}
below are not new.
They have become more or less  standard in the quantitative homogenization theory (see e.g. \cite{ Otto-Fisher, Shen-book, Armstrong-book}). 
Throughout this section, unless otherwise indicated,
we assume  $A$ is a matrix  satisfying  the ellipticity  conditions (\ref{ellipticity}) and
the periodicity condition  (\ref{periodicity}).
Let  $\chi (y) =(\chi_j^{\alpha\beta} (y) )\in H^1_{\loc}(\mathbb{R}^d)$ denote  the correctors for the operator,
\begin{equation}\label{op}
\mathcal{L}_\e =-\text{\rm div} (A(x/\e) \nabla ).
\end{equation}
More precisely, 
 for each $1\le j\le d$ and $1\le \beta\le m$,
$\chi_j^\beta = (\chi_j^{1\beta}, \chi_j^{2\beta}, \dots, \chi_j^{m\beta}) $  is the unique solution
of the following problem:
\begin{equation}\label{corrector}
\left\{
\aligned
& -\text{\rm div} (A\nabla \chi_j^\beta)=\text{\rm div}(A\nabla P_j^\beta) \quad \text{ in } \mathbb{R}^d,\\
 & \chi_j^\beta \text{ is 1 -periodic},\\
 &\int_Y \chi_j^\beta (y)\, dy=0.
 \endaligned
 \right.
 \end{equation}
 In (\ref{corrector}) we have used notation  $Y=[0, 1)^d$ and
 $P_j^\beta =y_j (0, \dots, 1, \dots, 0)$ with $1$ in the $\beta^{th}$ place.
 The homogenized operator is given by $\mathcal{L}_0=-\text{\rm div}(\widehat{A}\nabla)$,
 where $\widehat{A}=(\widehat{a}_{ij}^{\alpha\beta})$ and
 $$
 \widehat{a}_{ij}^{\alpha\beta}
 =\fint_Y \Big[
 a_{ij}^{\alpha\beta}
 + a_{ik}^{\alpha\gamma} \frac{\partial}{\partial y_k} \Big(\chi_j^{\gamma\beta} \Big) \Big]\, dy,
 $$
 where the repeated indices are summed.
 Let
 \begin{equation}\label{B}
 b_{ij}^{\alpha\beta}
 =a_{ij}^{\alpha\beta}
 + a_{ik}^{\alpha\gamma} 
 \frac{\partial}{\partial y_k} \big( \chi_j^{\gamma \beta}\big)
 -\widehat{a}_{ij}^{\alpha\beta}.
 \end{equation}
 Note that
 $$
 \int_Y b_{ij}^{\alpha\beta}\, dy=0 \quad 
 \text{ and } \quad
 \frac{\partial}{\partial y_i}
 \big( b_{ij}^{\alpha\beta} \big)=0.
 $$
 There exist 1-periodic functions  $\phi_{kij}^{\alpha\beta}\in H^1_{\loc}(\mathbb{R}^d) $ such that
 $\int_Y \phi_{kij}^{\alpha\beta} \, dy=0$,
 \begin{equation}\label{flux}
 b_{ij}^{\alpha\beta}
 =\frac{\partial}{\partial y_k} \big( \phi_{kij}^{\alpha\beta} \big)
 \quad \text{ and } \quad
 \phi_{kij}^{\alpha\beta}=-\phi_{ikj}^{\alpha\beta},
 \end{equation}
 where $1\le i, j, k\le d$ and $1\le \alpha, \beta\le m$.
 The functions $\phi=(\phi_{kij}^{\alpha\beta})$ are called flux or dual correctors.
 
Let $u_\e \in H^1(\Omega)$ be a weak solution of the Dirichlet problem,
\begin{equation}\label{DP-10}
\text{\rm div}(A(x/\e)\nabla u_\e)=0 \quad \text{ in }  \Omega \quad \text{ and } \quad  
u_\e=g \quad   \text{ on } \partial\Omega.
\end{equation}
Let $u_0$ be the solution of the homogenized problem,
\begin{equation}\label{DP-h-10}
\text{\rm div}(\widehat{A} \nabla u_0)=0 \quad \text{ in }  \Omega \quad \text{ and } \quad  
u_0=g \quad   \text{ on } \partial\Omega.
\end{equation}
Consider 
\begin{equation}\label{w}
w_\e
=u_\e -u_0 -\e \chi(x/\e) 
 (\eta_\e \nabla u_0),
\end{equation}
where $\eta_\e\in C_0^\infty(\Omega)$ is a cut-off function such that
$0\le \eta_\e\le 1$, $|\nabla \eta_\e |\le C/\e$,
$\eta_\e (x) =0$ if dist$(x, \partial\Omega)< 4\e$,
$\eta(x)=1$ if $x\in \Omega$ and dist$(x, \partial\Omega)\ge 5\e$.
Let 
\begin{equation}\label{Sigma}
\Sigma_t 
=\big\{ x\in \Omega: \text{\rm dist}(x, \partial\Omega) < t \big\},
\end{equation}
where $0<t< \text{\rm diam}(\Omega)$.
 
\begin{lemma}\label{app-lemma-1}
Suppose $A$ satisfies (\ref{ellipticity}) and (\ref{periodicity}).
Also assume that $A\in \text{\rm VMO} (\mathbb{R}^d)$ if $m\ge 2$ and $d\ge 3$.
Let $\Omega$ be a bounded Lipschitz domain with $1\le \text{diam}(\Omega)\le 10$.
Let $w_\e$ be given by (\ref{w}).
Then for any $0<\e<1$,
\begin{equation}\label{app-1-0}
\|\nabla w_\e \|_{L^2(\Omega)}
\le C \Big\{ \|\nabla u_0\|_{L^2(\Sigma_{5\e})}
+ \e \|\nabla^2 u_0\|_{L^2(\Omega\setminus \Sigma_{4\e)}} \Big\},
\end{equation}
where $C$ depends only on $d$, $m$,  $\mu$, the function $\rho$ in (\ref{VMO}) (if $m\ge 2$ and $d\ge 3$),  and the Lipschitz character of $\Omega$.
\end{lemma}

\begin{proof}
Note that $w_\e=0$ on $\partial\Omega$.
Let $\phi=\big( \phi_{kij}^{\alpha\beta}\big)$ be given by (\ref{flux}).
A direct computation shows that 
\begin{equation}\label{L-w}
\mathcal{L}_\e (w_\e)
= \text{\rm div}\Big\{  (A^\e-\widehat{A}) (1-\eta_\e)( \nabla u_0 ) \Big\}
+ \e \text{\rm div}
\Big\{ \big(\phi^\e + \chi^\e A^\e \big) 
\nabla \big(\eta_\e (\nabla u_0) \big) \Big\},
\end{equation}
where $A^\e=A(x/\e)$, $\chi^\e=\chi(x/\e)$ and $\phi^\e =\phi(x/\e)$. 
See e.g. \cite{KLS-2012, Shen-book}.
In the  case $m=1$ or $d=2$,
 the correctors $\chi$ and $\phi$ are bounded.
They are also bounded if $m\ge 2$ and  $d\ge 3$ under the assumption  $A\in \text{VMO}(\mathbb{R}^d)$.
As a result, (\ref{app-1-0}) follows from (\ref{L-w}) by the energy estimate.
\end{proof}

\begin{lemma}\label{app-lemma-3}
Let $A$ and $\Omega$ be the same as in Lemma \ref{app-lemma-1}.
Then there exists $\kappa\in (0, 1)$, depending only on $d$, $m$, $\mu$ and the Lipschitz character of $\Omega$, such that
\begin{equation}\label{app-3-0}
\| \nabla u_\e 
-\nabla u_0
-( \nabla \chi)^\e   \eta_\e (\nabla u_0)
\|_{L^2(\Omega)}
\le C \e^{\kappa}
\|g\|_{H^1(\partial\Omega)},
\end{equation}
where 
$(\nabla \chi)^\e =\nabla \chi(x/\e)$ and
$C$ depends only on $d$, $\mu$, $\mu$, the function $\rho$ in (\ref{VMO}) (if $m\ge 2$ and $d\ge 3$), and the Lipschitz character of $\Omega$.
\end{lemma}

\begin{proof}
Note that
$$
\nabla w_\e=\nabla u_\e-\nabla u_0
-(\nabla \chi)^\e \eta_\e (\nabla u_0)
-\e (\chi)^\e \nabla ( \eta_\e (\nabla u_0)).
$$
In view of Lemma \ref{app-lemma-1}, it suffices to prove that 
\begin{equation}\label{app-3-00}
\| \nabla u_0\|_{L^2(\Sigma_{5\e})}
+\e \|\nabla^2 u_0\|_{L^2(\Omega\setminus \Sigma_{4\e})}
\le C \e^\kappa \| g\|_{H^1(\partial\Omega)}.
\end{equation}
To this end, we choose  a function $G\in H^1(\Omega) $ such  that $G=g$ on $\partial\Omega$ and
$$
\int_\Omega |\nabla G|^2 [\text{dist}(x, \partial\Omega)]^{-2\kappa} \, dx \le C\| g\|_{H^1(\partial\Omega)}^2,
$$
where $\kappa\in (0, 1/2)$ is given by Theorem \ref{g-thm}.
Since $-\text{\rm div}(\widehat{A}\nabla (u_0-G))=\text{\rm div}(A\nabla G)$ in $\Omega$,
it follows from Theorem  \ref{g-thm} that
\begin{equation}\label{app-3-1}
\int_\Omega |\nabla u_0|^2 [\text{dist}(x, \partial\Omega)]^{-2\kappa} \, dx \le C\| g\|_{H^1(\partial\Omega)}^2.
\end{equation}
This implies that
$$
\|\nabla u_0\|_{L^2(\Sigma_{5\e})}
\le C \e^{\kappa} \| g\|_{H^1(\partial\Omega)}.
$$
Also, by the interior estimates for the elliptic systems with constant coefficients,
$$
\aligned
\int_{\Omega\setminus \Sigma_{4\e}} |\nabla^2 u_0|^2\, dx
 & \le C\int_{\Omega\setminus \Sigma_{3\e}}
\frac{|\nabla u_0 (x)|^2}
{ [ \text{\rm dist} (x, \partial\Omega)]^2}\, dx \\
&\le C \e^{2\kappa -2}
\| g\|_{H^1(\partial\Omega)}^2,
\endaligned
$$
where we have used (\ref{app-3-1}) for the last inequality.
This, together with (\ref{app-3-1}) and (\ref{app-1-0}), gives (\ref{app-3-00}).
\end{proof}

\begin{lemma}\label{app-lemma-5}
Let $A$ be a matrix satisfying the same conditions as in Lemma \ref{app-lemma-1}.
Let $u_\e\in H^1(B_{2r}) $ be a weak solution of $\text{\rm div}(A(x/\e)\nabla u_\e)=0$ in $B_{2r}$, where $r\ge 100\e$.
Then there exists $u_0\in H^1(B_{3r/2})$ such that
$\text{\rm div} (\widehat{A}\nabla u_0)=0$ in $B_{3r/2}$ such that
\begin{align}
\left(\fint_{B_{3r/2}} |\nabla u_0|^2\right)^{1/2}
 & \le C \left(\fint_{B_{2r}} |\nabla u_\e|^2 \right)^{1/2}, \label{app-5-0}\\
\left(\fint_{B_r}
|\nabla u_\e -\nabla u_0
-( \nabla \chi)^\e  (\nabla u_0)  |^2 \right)^{1/2}
 & \le C \left( \frac{\e}{r} \right)^{\kappa}
\left(\fint_{B_{2r}}
|\nabla u_\e|^2 \right)^{1/2},\label{app-5-1}
\end{align}
where $C$ depends only on $d$, $m$, $\mu$, and the function $\rho$ in (\ref{VMO}) (if $m\ge 2$ and $d\ge 3$).
\end{lemma}

\begin{proof}
By rescaling we may assume $r=1$. There exists $t\in (3/2,2)$ such that
$$
\int_{\partial B_t}
|\nabla u_\e|^2 \, d\sigma 
\le C \int_{B_2} |\nabla u_\e|^2\, dx.
$$
This follows readily by using the polar coordinates.
Let $u_0\in H^1(B_t)$ be the solution of
$\text{\rm div}(\widehat{A}\nabla u_0)=0$ in $B_t$
with $u_0=u_\e -k $ on $\partial B_t$, where $k=\fint_{\partial B_t} u_\e$.
By the energy estimate,
$$
\int_{B_t} |\nabla u_0|^2\, dx \le C \int_{B_t} |\nabla u_\e|^2\, dx,
$$
which yields  (\ref{app-5-0}) with  $r=1$.
To see (\ref{app-5-1}), we apply Lemma \ref{app-lemma-3} with $\Omega=B_t$.
Since $\eta_\e=1$ in $B_1$ and $\| u_\e -k\|_{H^1(\partial B_t)}
\le C \|\nabla u_\e\|_{L^2(\partial B_t)}$, we obtain (\ref{app-5-1}).
\end{proof}

\begin{lemma}\label{app-lemma-6}
Assume $A$ satisfies the same conditions as in Lemma \ref{app-lemma-1}.
Let $\Omega$ be a bounded Lipschitz domain.
Let $u_\e\in H^1(B_{4r}\cap \Omega)$ be a weak solution of $\text{\rm div}(A(x/\e)\nabla u_\e)=0$ in $B_{4r}\cap \Omega$ with
$u_\e =0$ on $B_{4r}\cap \Omega$, where $B_r=B(x_0, r)$ and  $x_0 \in \partial\Omega$.
Assume $ 100 \e< r< cr_0$, where $r_0=\text{\rm diam}(\Omega)$.
Then there exists $u_0\in H^1 (B_{3r/2} \cap \Omega)$ such that
$\text{\rm div}(\widehat{A}\nabla u_0)=0$ in $B_{3r/2} \cap \Omega$,
$u_0=0$ on $B_{3r/2}\cap \partial\Omega$, and
\begin{align}
\left(\fint_{B_{3r/2} \cap \Omega}  |\nabla u_0|^2\right)^{1/2}
 & \le C \left(\fint_{B_{2r}\cap \Omega } |\nabla u_\e|^2 \right)^{1/2}, \label{app-6-0}\\
\left(\fint_{B_r\cap \Omega }
|\nabla u_\e -\nabla u_0
- ( \nabla \chi)^\e (\nabla u_0)\eta_\e   |^2 \right)^{1/2}
 & \le C \left( \frac{\e}{r} \right)^{\kappa}
\left(\fint_{B_{2r}\cap \Omega }
|\nabla u_\e|^2 \right)^{1/2},\label{app-6-1}
\end{align}
where $C$ depends only on $d$, $m$, $\mu$, the function $\rho$ in (\ref{VMO}) (if $m\ge 2$ and $d\ge 3$)
and the Lipschitz character of $\Omega$.
\end{lemma}

\begin{proof}
The proof is similar to that of Lemma \ref{app-lemma-5}.
By rescaling we may assume $r=1$.
In the place of $B_t$ we use $B(x_0, t)\cap \Omega$.
We omit the details.
\end{proof}

\begin{lemma}\label{lemma-H-L}
Assume  $A$ satisfies conditions (\ref{ellipticity})-(\ref{periodicity}) and $A\in \text{\rm VMO}(\mathbb{R}^d)$.
Let $\omega$ be an $A_1$ weight and $\Omega$ a bounded Lipschitz domain with $\text{\rm diam}(\Omega)=1$.
Suppose that Condition (2) in Theorem \ref{main-thm-1}  holds 
in  $\Omega$ with
weight $\omega$ for the homogenized operator $\mathcal{L}_0$.
Let $B_0=B(x_0, r_0)$ with the properties that  either $ 4B_0\subset \Omega$ or $x_0\in \partial\Omega$ and $0<r_0<c_0$.
Then 
\begin{equation}\label{L-e-1}
\int_{B_0\cap \Omega}
\Big\{
\mathcal{M}_{4B_0}^{\e} (|\nabla u_\e| \chi_{4B_0\cap \Omega} ) \Big\}^2
\omega \, dx
\le C \int_{4B_0\cap \Omega} |\nabla u_0|^2 \fint_{B_0} \omega,
\end{equation}
where
$\text{\rm div}(A(x/\e)\nabla u_\e)=0$ in $4B_0\cap \Omega$ and
$u_\e=0$ on $4B_0\cap \Omega$ (if $x_0\in \partial\Omega$).
\end{lemma}

\begin{proof}

We may also assume $0<\e< c_0$ and $c_0$ is small.
The case $\e\ge c_0$  is trivial.
We consider the case $x_0\in \partial\Omega$ (the interior case $4B_0\subset \Omega$ is similar).
We apply Theorem \ref{real-thm-b} with $F=|\nabla u_\e|$ and $f=0$.
Fix a  large  constant $L>1$.
Let $B=B(y_0, r)$ be a ball with the properties that $r\ge L\e$,
$|B|\le c_1 |B_0|$
and that  either $y_0\in 2B_0\cap\partial  \Omega$ or $4B\subset 2B_0\cap \Omega$.
Again, we only consider the case $y_0\in  2B_0\cap \partial \Omega$.
Note   that $\text{\rm div}(A(x/\e)\nabla u_\e) =0$ in $8B\cap \Omega$ and
$u_\e=0$ on $8B\cap \partial\Omega$.
Let $u_0$ be the solution of $\text{\rm div}(\widehat{A}\nabla u_0)=0$ in $ 4B \cap \Omega$,
constructed in Lemma \ref{app-lemma-6}.
Define
$$
F_B=|\nabla u_\e -\nabla u_0 -( \nabla \chi)^\e (\eta_\e (\nabla u_0) )  |
\quad \text{ and } \quad
R_B= |\nabla u_0 +(\nabla \chi)^\e  (\eta_\e (\nabla u_0)  ) |.
$$
It follows from (\ref{app-6-1}) that
\begin{equation}\label{app-10-2}
\left(\fint_{2B\cap\Omega }
|F_B|^2\right)^{1/2}
\le C \left(\frac{\e}{r}\right)^\kappa
\left(\fint_{6B\cap \Omega} |F|^2\right)^{1/2}.
\end{equation}
By using H\"older's inequality for $F_B$ and the reverse H\"older estimate for $\nabla u_\e$, we obtain 
\begin{equation}\label{app-10-3}
\left(\fint_{2B\cap\Omega }
|F_B|^{p_0}\right)^{1/p_0}
\le C \left(\frac{\e}{r}\right)^\kappa
\left(\fint_{8B\cap \Omega} |F|^{p_0} \right)^{1/p_0},
\end{equation}
where $1< p_0<2$ depends only on $d$, $\mu$ and the Lipschitz character of $\Omega$.

To verify the condition for $R_B$, we note that the VMO condition on $A$ implies that
$|\nabla \chi|\in L^q(Y)$ for any $q>2$.
Since  $r\ge \e$, it follows that 
$$
\left(\fint_{2B\cap \Omega} 
|( \nabla \chi )^\e|^q \right)^{1/q}
\le C \left(\int_Y |\nabla \chi|^q\right)^{1/q}<\infty
$$
for any $2<q<\infty$.
Hence, by H\"older's inequality,  for $2<p_1<p_2$,
$$
\aligned
\left(\fint_{2B\cap \Omega}
|R_B \omega^{1/2} |^{p_1} \right)^{1/p_1}
 & \le C \left(\fint_{2B\cap \Omega} 
|( \nabla u_0) \omega^{1/2}|^{p_1}  (1+| (\nabla \chi)^\e|)^{p_1} \right)^{1/p_1}\\
& \le C \left(\fint_{2B\cap \Omega} 
|( \nabla u_0) \omega^{1/2}|^{p_2}  \right)^{1/p_2}\\
&\le C \fint_{4B\cap \Omega}
|(\nabla u_0)\omega^{1/2}|
\endaligned
$$
where $p_2>2$ is the exponent $p$ in (\ref{eqv-1}) for the homogenized operator $\mathcal{L}_0$.
It follows that
$$
\aligned
\left(\fint_{2B\cap \Omega}
|R_B \omega^{1/2} |^{p_1} \right)^{1/p_1}
&\le C \left(\fint_{4B\cap \Omega}
|\nabla  u_0 |^2 \right)^{1/2} \left(\fint_B \omega \right)^{1/p_1}\\
&\le
C \left(\fint_{6B\cap \Omega}
|F|^2 \right)^{1/2} \left(\fint_B \omega \right)^{1/p_1}\\
&\le
C \left(\fint_{8B\cap \Omega}
|F|^{p_0}  \right)^{1/p_0} \left(\fint_B \omega \right)^{1/p_1}.
\endaligned
$$
This gives the condition (\ref{real-2b}) in Theorem \ref{real-thm-b}.
We now choose $L>1$ so large that  $CL^{-\kappa} \le \eta_0$, where $\eta_0>0$ is given  in Theorem \ref{real-thm-b}.
By Remark \ref{remark-L-b} we obtain (\ref{L-e-1}) with $\mathcal{M}_{4B_0}^{L\e}$ in the 
place of $\mathcal{M}^\e_{4B_0}$.
However, it is not hard to see that $\mathcal{M}_{4B_0}^\e
\le L^d  \mathcal{M}_{4B_0}^{L\e}$.
This completes the proof.
\end{proof}

We are now ready to give the proof of Theorem \ref{main-thm-2}.

\begin{proof}[\bf Proof of Theorem \ref{main-thm-2}]

By dilation we may assume diam$(\Omega)=1$.
We will verify Condition (2) in Theorem \ref{main-thm-1} for the operator $\mathcal{L}_\e$, assuming that the same condition
holds for operators $-\text{\rm div}(\overline{A}\nabla )$ with constant coefficients, where either $\overline{A}=\widehat{A}$ or $\overline{A}$
is obtained from $A$ by averaging over a ball.

Let $B_0=B(x_0, r_0)$, where $ 0< r_0< c_0$ and $c_0>0$  is sufficiently small.
We assume $x_0\in \partial \Omega$ (the interior case $4B_0\subset \Omega$ is similar).
Suppose that $\text{\rm div} (A(x/\e)\nabla u_\e)=0$ in $4B_0\cap \Omega$ and
$u_\e=0$ on $4B_0\cap \partial\Omega$.
We  need  to show that 
\begin{equation}\label{app-10-0}
\left(\fint_{B_{0} \cap \Omega}
|\nabla u_\e|^2\omega\right)^{1/2}
\le C
 \left(\fint_{ 4B_0\cap \Omega}
|\nabla u_\e|^2 \right)^{1/2}
\left(\fint_{B_{0}} \omega\right)^{1/2}.
\end{equation}
We first observe that if $\e\ge c_0$, the estimate (\ref{app-10-0}) follows  directly Theorem \ref{thm-local}.
To see this, we note that 
$$
\rho (r; A^\e)=\rho (r/\e; A) \le \rho(1; A)<\infty,
$$
where $\rho$ is  given by (\ref{VMO}).  Thus, by Theorems \ref{main-thm-1} and \ref{thm-local},
(\ref{app-10-0}) holds uniformly in $\e\ge c_0$.

Suppose $0<\e<c_0$.
It follows from the proof of Theorem \ref{thm-local} that if $B=B(y_0, r)$, where $r= \e$ and
either $y_0\in 2B_0\cap  \partial\Omega$  or $4B\subset \Omega$, then
\begin{equation}\label{s-e-1}
\fint_{B \cap \Omega}
|\nabla u_\e|^2 \omega\, dx
\le C \left (  \fint_{2B\cap \Omega} |\nabla u_\e|\, dx\right)^2
\fint_B \omega.
\end{equation}
This implies that
\begin{equation}\label{s-e-2}
\int_{B \cap \Omega}
|\nabla u_\e|^2 \omega\, dx
\le C \int_{2B\cap \Omega} 
\big\{ \mathcal{M}^\e_{4B_0} ( |\nabla u_\e| \chi_{4B_0\cap \Omega}) \big\}^2 \omega \, dx,
\end{equation}
where the operator $\mathcal{M}^\e_{4B_0} $ is defined by (\ref{maximal-e}).
By s simple covering argument we obtain 
\begin{equation}\label{s-e-3}
\aligned
\int_{B_0 \cap \Omega}
|\nabla u_\e|^2 \omega\, dx
 & \le C \int_{2B_0\cap \Omega} 
\big\{ \mathcal{M}^\e_{4B_0} ( |\nabla u_\e| \chi_{4B_0\cap \Omega}) \big\}^2 \omega \, dx\\
&\le C \int_{4B_0\cap \Omega} |\nabla u_\e|^2
\fint_{B_0} \omega,
\endaligned
\end{equation}
where we have used Lemma \ref{lemma-H-L} for the last  inequality.
\end{proof}

We end this section with a result for $C^1$ domains.

\begin{thm}\label{thm-4.30}
Suppose that $A=A(y)$ satisfies the ellipticity condition (\ref{ellipticity}), the periodicity condition (\ref{periodicity})
and $A\in \text{\rm VMO}(\mathbb{R}^d)$.
Let $\Omega$ be a bounded $C^1$ domain.
Let $u_\e\in H^1_0(\Omega)$ be a weak solution of (\ref{DP-0})
with $f\in L^\infty(\Omega)$.
Then for any $A_1$ weight $\omega$,
\begin{equation}\label{4.30-0}
\int_{\Omega} 
|\nabla u_\e|^2\omega^{\pm 1} \, dx
\le C \int_\Omega |f|^2\omega^{\pm 1} \, dx,
\end{equation}
where $C$ depends only  on $d$, $m$, the function $\rho$ in (\ref{VMO}), the $A_1$ constant of $\omega$,  and $\Omega$. 
\end{thm}

\begin{proof}
Let $B=B(x_0, r)$ with a ball with the properties that $|B|\le c_1 |\Omega|$
and either $4B\subset \Omega$ or $x_0\in \partial\Omega$.
Let $u_\e$ be a weak solution of
$\text{\rm div}(A(x/\e)\nabla u_\e)=0$ in $4B\cap \Omega$ with
$u_\e=0$ on $4B\cap \partial\Omega$ (in the case $x_0\in \partial\Omega$).
It follows from \cite{Shen-2008-W, Shen-2017-W} that
$$
\left(\fint_{B\cap \Omega}
|\nabla u_\e|^p \right)^{1/p}
\le C \left(\fint_{2B\cap \Omega} 
|\nabla u_\e|^2 \right)^{1/2}
$$
for any $p>2$.
By H\"older's inequality and (\ref{RH}),
this gives  the inequality (\ref{cond-1})  for any $\omega\in A_1(\mathbb{R}^d)$.
By Theorem \ref{main-thm-1}  we obtain  (\ref{4.30-0}) for $\omega^{+1}$.
Since $A^*$ satisfies the same conditions as $A$, the case $\omega^{-1}$ follows by
duality.
\end{proof}

If $\Omega$ is a Lipschitz domain,
the  inequality (\ref{4.30-0}) may not be true for all $A_1$ weights, as 
this would imply the $W^{1, p}$ estimate $\|\nabla u_\e\|_{L^p(\Omega)}
\le C_p \| f\|_{L^p(\Omega)}$ for  $1<p<\infty$, by a general extrapolation result.

%%%%%%%%%%%%%%%%%%%%%%%%%%%%%%%%%%%

\section{Proof of Theorem \ref{main-thm-3}}\label{section-sp}

In this section we consider  the weight $
\omega_\sigma   (x)=[\text{\rm dist} (x, \partial\Omega)]^\sigma $,
where $-1< \sigma <1$ and $\Omega$ is a  Lipschitz domain.
Recall that  $w_\sigma $ is an $A_1$ weight if $-1<\sigma  \le 0$,
and $\omega_\sigma  \in A_p(\mathbb{R}^d)$ for $p>1$ if $-1< \sigma < p-1$.

\begin{lemma}\label{lemma-lip-1}
Let  $A$ be  a constant matrix satisfying  (\ref{ellipticity}) and
 $\Omega$  a bounded Lipschitz domain.
Also assume that $A^*=A$ if $m\ge 2$.
Let $u\in H^1_0(\Omega) $ be a weak solution of
$-\text{\rm div}(A\nabla u) =\text{\rm div}(f)$ in $\Omega$,
where $f\in L^\infty(\Omega)$.
Then, for any $-1< \sigma <1$,
\begin{equation}\label{lip-1}
\int_\Omega 
|\nabla u |^2 [\text{\rm dist}(x, \partial\Omega)]^\sigma\, dx
\le C\int_\Omega
 |f|^2 [\text{\rm dist}(x, \partial\Omega)]^\sigma \, dx,
\end{equation}
 where
$C$ depends only on $d$, $m$, $\mu$, $\sigma$ and the Lipschitz character of $\Omega$.
\end{lemma}

\begin{proof}
We may assume $-1<\sigma <0$; the case $0<\sigma <1$ follows by duality.
Since $\omega_\sigma\in A_1(\mathbb{R}^d)$,
 by Theorems  \ref{main-thm-1},
  we only need to check the condition (\ref{cond-1}).

Let $B=B(x_0, r)$ be a ball with $|B|\le c_1|\Omega|$.
The case $4B\subset \Omega$ is trivial.
To treat the case $x_0\in \partial\Omega$,
we assume that $\text{\rm div}({A} \nabla u)=0$ in $4B \cap \Omega$ and
$u=0$ on $ 4B \cap \partial\Omega$.
Without loss of generality, we may  assume that
$$
\Omega\cap 10B
=\big\{
(x^\prime,  x_d)\in \mathbb{R}^d: \
x_d >\psi (x^\prime) \big\},
$$
where $\psi: \mathbb{R}^{d-1} \to \mathbb{R}$ is a Lipschitz function with $\psi (0)=0$.
Let
$$
(\nabla u)_r^*(x^\prime, \psi(x^\prime))
=\sup \Big\{
|\nabla u(x, \psi(x^\prime) +s)|: \ 0<s<r \Big\}.
$$
It follows from the nontangential-maximal-function estimates (see e.g. \cite{Kenig-book, Shen-book}) that
$$
\int_{B\cap \partial\Omega}
|(\nabla u)^*_{c r} |^2\, d\sigma
\le \frac{C}{r}
\int_{2B\cap \Omega }
|\nabla u|^2\, dx.
$$
We point out  that  if $m\ge 2$,
the estimate above requires the condition $A^*=A$.
In the scalar case $m=1$, the symmetry condition is not needed,
as one may write $\text{\rm div}(A\nabla u) =(1/2)  \text{\rm div} ((A+A^*)\nabla u))$.
Hence,
$$
\aligned
\fint_{B\cap \Omega}
|\nabla u|^2 \omega_\sigma \, dx
&\le 
C r^{1+ \sigma -d}
\int_{B \cap  \partial\Omega} 
|(\nabla u)^*_{cr} |^2\, d\sigma
+ Cr^{\sigma }  \fint_{2B\cap \Omega}
|\nabla u|^2\, dx\\
&\le Cr^{\sigma }  \fint_{2B\cap \Omega}
|\nabla u|^2\, dx\\
& \le C \fint_{2B\cap \Omega} |\nabla u|^2\, dx
\cdot \fint_B \omega_\sigma,
\endaligned
$$
which gives (\ref{cond-1})
\end{proof}

\begin{proof}[\bf Proof of Theorem \ref{main-thm-3}]

By duality we may assume $-1<\sigma<0$.
Since $\omega_\sigma \in A_1(\mathbb{R}^d)$,
in view of Theorem \ref{main-thm-2},
it suffices to establish the weighted estimate (\ref{w-e-1}) with $\omega=\omega_\sigma$
for weak solutions in $H^1_0(\Omega)$  of
$-\text{\rm div} (\overline{A}\nabla u)=\text{\rm div}(f)$ in $\Omega$,
where  the constant matrix $\overline{A}$ is either $\widehat{A}$ or obtained from $A$ by averaging over a ball.
But this is already done in Lemma \ref{lemma-lip-1}. Indeed,
 in both case, $\overline{A}$ satisfies (\ref{ellipticity}).
Also, since $A$ is symmetric for $m\ge 2$, 
so  is  $\overline{A}$.
\end{proof}

We end this section with a weighted inequality with the weight $\omega_\alpha$
for any matrix satisfying (\ref{ellipticity}).
The inequality was used in the proof of Lemma \ref{app-lemma-3}.

\begin{thm}\label{g-thm}
Let $A$   be a matrix satisfying (\ref{ellipticity}) and $\Omega$ a bounded Lipschitz domain.
Let $u\in H^1_0(\Omega)$ be a weak solution of
$-\text{\rm div}(A\nabla u) =\text{\rm div} (f)$ in $\Omega$,
where $f\in L^2(\Omega)$.
Then there exists $\kappa\in (0, 1/2)$, depending only on $d$, $m$, $\mu$ and
the Lipschitz character of $\Omega$, such that  for $|\sigma  |\le  2\kappa$,
\begin{equation}\label{g-w}
\int_\Omega
|\nabla u|^2 [ \text{\rm dist} (x, \partial\Omega)]^\sigma \, dx
\le  C \int_\Omega
|f|^2 [ \text{\rm dist} (x, \partial\Omega)]^\sigma \, dx,
\end{equation}
where 
$C>0 $ depend only on $d$, $m$,  $\mu$ and the Lipschitz character of $\Omega$.
\end{thm}

\begin{proof}
The result  is probably well known. We provide a proof here for reader's convenience.
By duality we may assume $\sigma<0$.
The proof uses Hardy's  inequality,
\begin{equation}\label{Hardy}
\int_\Omega
|u|^2 [\text{\rm dist} (x, \partial\Omega) ]^{-2} \, dx
\le C \int_\Omega |\nabla u|^2 \, dx
\end{equation}
for any $u\in H_0^1(\Omega)$,
where   $C$ depends only on $d$ and the Lipschitz character of $\Omega$.
Now, let  $u\in H^1_0(\Omega)$ be a weak solution of
$-\text{\rm div}(A\nabla u) =\text{\rm div} (f)$ in $\Omega$,
where $f\in L^2(\Omega)$.
Let
$
\psi_t  (x) = \text{\rm dist} (x, \partial\Omega) + t,
$
where $t>0$.
Note that $|\nabla \psi_t |\le 1$.
Using 
$$
\int_{\Omega} A\nabla u \cdot \nabla (u \psi_t^{2\sigma })\, dx
=-\int_\Omega f \cdot \nabla (u \psi_t^{2\sigma } )\, dx
$$
and the Cauchy inequality,
we obtain
\begin{equation}\label{g-1}
\aligned
\int_\Omega |\nabla u|^2 \psi_t^{2\sigma } \, dx
 & \le C \int_\Omega |u|^2 |\nabla \psi_t^{\sigma }|^2\, dx
+ C \int_\Omega
|f|^2 \psi_t^{2\sigma}\,dx\\
  & \le C |\sigma|^2  \int_\Omega |u|^2 | \psi_t^{\sigma-1 }|^2\, dx
+ C \int_\Omega
|f|^2 \psi_t^{2\sigma}\,dx,
\endaligned
\end{equation}
where $C$ depends only on $\mu$.
Since $u\psi_t^\sigma \in H_0^1(\Omega)$, by Hardy's inequality (\ref{Hardy}),
$$
\aligned
 \int_\Omega |u|^2 | \psi_t^{\sigma-1 }|^2\, dx
 &\le \int_\Omega |u \psi_t ^\sigma|^2 [\text{\rm dist} (x, \partial\Omega) ]^{-2}\, dx\\
 & \le C \int_\Omega 
 |\nabla (u \psi_t^\sigma)|^2\, dx\\
 &\le C \int_\Omega |\nabla u|^2 \psi_t^{2\sigma} \, dx
 + C |\sigma|^2 \int_\Omega |u |^2 | \psi_t ^{\sigma -1}|^2\, dx,
\endaligned
$$
where $C$ depends only on $d$ and  the Lipschitz character of $\Omega$.
Thus, if  $C|\sigma|^2\le (1/2)$, 
$$
 \int_\Omega |u|^2 | \psi_t^{\sigma-1 }|^2\, dx
\le  C \int_\Omega |\nabla u|^2 \psi_t^{2\sigma} \, dx.
$$
This, together with (\ref{g-1}), gives
$$
\int_\Omega |\nabla u|^2 \psi_t^{2\sigma } \, dx
\le  C |\sigma|^2 \int_\Omega |\nabla u|^2 \psi_t^{2\sigma } \, dx
+ C \int_\Omega
|f|^2 \psi_t^{2\sigma}\,dx.
$$
Again, if $ C |\sigma|^2\le (1/2)$, then 
$$
\int_\Omega |\nabla u|^2 \psi_t^{2\sigma } \, dx
\le C \int_\Omega
|f|^2 \psi_t^{2\sigma}\,dx.
$$
As a result, we have proved that if $C|\sigma|^2 \le (1/2)$, where $\sigma<0$ and $C$ depends on  $d$, $\mu$ and the Lipschitz character of $\Omega$,
$$
\int_\Omega |\nabla u|^2 \psi_t^{2\sigma } \, dx
\le C \int_\Omega
|f|^2 [\text{\rm dist}(x, \partial\Omega) ]^{2\sigma}\,dx.
$$
By letting $t\to 0^+$ and using Fatou's Lemma we obtain (\ref{g-w}).
\end{proof}

 \bibliographystyle{amsplain}
 
\bibliography{S20201.bbl}

\bigskip

\begin{flushleft}

Zhongwei Shen,
Department of Mathematics,
University of Kentucky,
Lexington, Kentucky 40506,
USA.

E-mail: zshen2@uky.edu
\end{flushleft}

\bigskip

\end{document}